\documentclass[reqno,11pt]{amsart}

\usepackage{xypic,hyperref,todonotes,comment,tikz}
\usetikzlibrary{arrows,calc,decorations.markings,math,arrows.meta}
\usepackage{graphics,amssymb,mathtools}
\usepackage{latexsym}
\usepackage{mathrsfs}
\xyoption{curve}

\usepackage{color}
\definecolor{page_color}{HTML}{ffcac0}
\definecolor{color1}{HTML}{880500}
\definecolor{color2}{rgb}{.20,.60,.22}
\definecolor{color3}{rgb}{0,.40,.80}

\usepackage{hyperref}
\hypersetup{ 
colorlinks,
citecolor=color1,
filecolor=blue,
linkcolor=black,
urlcolor=color1
}

\newtheorem{lemma}{Lemma}[section]
\newtheorem{proposition}[lemma]{Proposition}
\newtheorem{theorem}[lemma]{Theorem}
\newtheorem{corollary}[lemma]{Corollary}
\newtheorem{question}[lemma]{Question}

\newtheorem*{theoremA}{Theorem}

\theoremstyle{definition}

\newtheorem{definition}[lemma]{Definition}
\newtheorem{remark}[lemma]{Remark}

\newcommand{\msc}[1]{\mathscr{#1}}
\newcommand{\mfk}[1]{\mathfrak{#1}}
\newcommand{\mcl}[1]{\mathcal{#1}}
\newcommand{\mbb}[1]{\mathbb{#1}}

\newcommand{\mbf}[1]{\mathbf{#1}}

\DeclareMathOperator{\Hom}{Hom}
\DeclareMathOperator{\Ext}{Ext}

\DeclareMathOperator{\Spec}{Spec}
\DeclareMathOperator{\Proj}{Proj}
\DeclareMathOperator{\Rep}{Rep}
\DeclareMathOperator{\rep}{rep}
\DeclareMathOperator{\supp}{supp}
\DeclareMathOperator{\res}{res}
\DeclareMathOperator{\Sym}{Sym}
\DeclareMathOperator{\Stab}{Stab}
\DeclareMathOperator{\stab}{stab}

\newcommand{\ot}{\otimes}
\newcommand{\bpsi}{\boldsymbol{\psi}}
\newcommand{\ploc}{\boldsymbol{\psi}\text{\rm-loc}}
\newcommand{\D}{\mathfrak{D}}
\renewcommand{\O}{\mathscr{O}}
\renewcommand{\1}{\mathbf{1}}
\renewcommand{\hat}{\widehat}

\newcommand{\nocontentsline}[3]{}
\newcommand{\tocless}[2]{\bgroup\let\addcontentsline=\nocontentsline#1{#2}\egroup}

\title[]{Support theory for Drinfeld doubles of some infinitesimal group schemes}
\date{\today}

\author{Eric M.\ Friedlander}
\author{Cris Negron}

\email{ericmf@usc.edu}

\email{cnegron@usc.edu}

\begin{document}

\maketitle

\begin{abstract}
Consider a Frobenius kernel $G$ in a split semisimple algebraic group, in very good characteristic.  We provide an analysis of support for the Drinfeld center $Z(\rep(G))$ of the representation category for $G$, or equivalently for the representation category of the Drinfeld double of $kG$.  We show that thick ideals in the corresponding stable category are classified by cohomological support, and calculate the Balmer spectrum of the stable category of $Z(\rep(G))$.  We also construct a $\pi$-point style rank variety for the Drinfeld double, identify $\pi$-point support with cohomological support, and show that both support theories satisfy the tensor product property.  Our results hold, more generally, for Drinfeld doubles of Frobenius kernels in any smooth algebraic group which admits a quasi-logarithm, such as a Borel subgroup in a split semisimple group in very good characteristic.
\end{abstract}

\tocless\section{}

In this paper we provide an in depth analysis of support theory for the Drinfeld double of a Frobenius kernel $G=\mbb{G}_{(r)}$ in a sufficiently nice algebraic group $\mbb{G}$.  Equivalently, we study support for the Drinfeld center of the representation category $\rep(G)$.  As indicated in the abstract, we calculate the Balmer spectrum of thick prime ideals in the stable category of representations for the double, classify thick ideals in the stable category, and construct $\pi$-point style rank varieties for representations.  Our rank variety construction is, of course, in line with that of Suslin-Friedlander-Bendel and Friedlander-Pevtsova \cite{suslinfriedlanderbendel97,friedlanderpevtsova05,friedlanderpevtsova07}.
\par

The present study occupies a somewhat unique position in the literature in that it is among the first semi-complete analyses of support for a class of ``properly quantum" finite tensor categories (cf.\ \cite[\S 4.1]{vashaw}).  By properly quantum here we mean braided, but non-symmetric.  The quantum setting is of particular interest from a physical perspective, where non-triviality of the braiding is used to encode certain geometric information in the guises of both topological and conformal field theories (see e.g. \cite{reshetikhinturaev91,fendley21,brochierjordansafronovsnyder,gannonnegron,koshidakytola}).
\par

From a purely mathematical perspective, support varieties have been employed to study various structural aspects of representations of groups and Hopf algebras.  The stratification they provide for various stable module categories was presaged by Quillen's stratification \cite{quillen71,quillen71II} of the spectrum of the cohomology of finite groups.  Indeed, cohomology (including $\Ext$-groups) plays a central role in the formulation of support theories, revealing a surprising wealth of information about representations.  Although 
the cohomology of a Hopf algebra $A$ does not depend upon the coproduct of $A$, the tensor product certainly does and the behavior of tensor products is a fundamental underpinning of many applications of representation theory.  Consequently, ``the tensor product property" for a suport theory $V \mapsto \supp(V)$ asserting that $\supp(V\otimes W) = \supp(V) \cap \supp(W)$ is of considerable interest.
\par 

As mentioned above, this text is dedicated to an analysis of support for the Drinfeld center $Z(\rep(G))$ of the representation category of an infinitesimal group scheme $G$.  The center $Z(\rep(G))$ can be understood as the universal braided tensor category which admits a central tensor functor to $\rep(G)$, in the sense of \cite[Definition 2.1]{bezrukavnikov04}.  There are, however, a number of more explicit presentations of the center.  For example, one can identify $Z(\rep(G))$ with the category $\operatorname{Coh}(G)^G$ of $ad$-equivariant sheaves on $G$.  Or, even more concretely, $Z(\rep(G))$ is identified with the representation category of the smash product
\[
D(G):=\O(G)\#_{\rm ad} kG
\]
of the algebra of functions on $G$ with the group ring of $G$.  The algebra $D(G)$ is called the \emph{Drinfeld double}, or \emph{quantum double}, of the group ring $kG$.  For more details one can see Section \ref{sect:Z} below.
\par

The Drinfeld center construction plays an essential role in studies of tensor categories and, as suggested above, in related studies in mathematical physics.  The important point here is that, unlike classical (symmetric) tensor categories, such as $\rep(G)$ itself, $Z(\rep(G))=\rep(D(G))$ is highly \emph{non-}symmetric, and so behaves more like a \emph{quantum} group than a classical group.  In particular, the Drinfeld center is what is called a nonsemisimple modular tensor category.  For indications of the particular relevance of \emph{cohomology} in emergent studies of conformal and topological field theories one can consider the texts \cite{lentneretal,schweigertwoike,costellocreutziggaitto19,gukovetal}, for example.
\par

Let us now turn to the specifics of this paper.  For the remainder of the introduction we fix a field $k$ of prime characteristic $p$, and consider the following:
\begin{itemize}
\item[-] \emph{Fix $G$ to be the $r$-th Frobenius kernel in a split semisimple algebraic group $\mbb{G}$, in very good characteristic}.\vspace{1mm}
\item[-] \emph{Fix $\D=D(G)$ to be the corresponding Drinfeld double for $kG$}.
\end{itemize}
Here $r$ is arbitrary, so that we are considering the family of normal subgroups $\mbb{G}_{(r)}$ in $\mbb{G}$.
\par

For an explicit example, one could consider $\mbb{G}$ to be $\operatorname{SL}_n(k)$ in odd characteristic $p$ which does not divide $n$, or the symplectic group $\operatorname{Sp}_{2n}(k)$ in arbitrary odd characteristic.  We note that all of the results listed below hold more generally when $\mbb{G}$ is replaced by an arbitrary smooth algebraic group over $k$ which admits a quasi-logarithm (see Section \ref{sect:qlog} for a definition).
\par

We recall the notion of cohomological support: For a finite-dimensional Hopf algebra $A$, and any $A$-representation $V$, we let $|A|$ denote the projective spectrum of cohomology, and $|A|_V$ denote the associated cohomological support space
\[
|A|=\Proj\Ext^\ast_A(k,k),\ \ |A|_V=\operatorname{Supp}_{|A|}\Ext^\ast_A(V,V)^\sim.
\]
Here $\Ext^\ast_A(V,V)$ inherits a graded module structure over $\Ext^\ast_A(k,k)$ via the tensor structure on $\rep(A)$, and $\Ext^\ast_A(V,V)^\sim$ denotes the associated sheaf on the projective spectrum.
\par

As a first point, we prove the following.

\begin{theoremA}[\ref{thm:tensor}]
Consider $G$ as above, with corresponding Drinfeld double $\D$.  Cohomological support for $\D$ satisfies the tensor product property.  That is to say, for finite-dimensional $\D$-representations $V$ and $W$ we have
\begin{equation}\label{eq:131}
|\D|_{(V\ot W)}=|\D|_V\cap |\D|_W.
\end{equation}
\end{theoremA}

From the perspective of tensor triangular geometry (e.g.\ \cite{balmer10II,bensoniyengarkrause11II}), Theorem \ref{thm:tensor} indicates that cohomological support may be used to ``structure" both the derived and stable categories of representations for the double $\D$.  We elaborate on this point, and also on our findings in this direction.
\par

Recall that the stable category $\stab(\D)$ for $\D$ is the quotient of $\rep(\D)$ by the ideal of all morphisms which factor through a projective.  This category inherits a triangulated structure from the abelian structure on $\rep(\D)$, and a monoidal structure from the monoidal structure on $\rep(\D)$.  Also, by a thick ideal in $\stab(\D)$ we mean a thick subcategory--and in particular a full triangulated subcategory--which is stable under the tensor action of $\stab(\D)$ on itself.  Finally, by a specialization closed subset in $|\D|$, we mean a subset $\Theta\subset |\D|$ which contains the closure $\bar{x}\subset \Theta$ of any point $x\in \Theta$.  We prove the following.

\begin{theoremA}[\ref{thm:thick_id}]
Cohomological support provides an order preserving bijection
\[
\begin{array}{c}
\{\text{\rm Specialization closed subsets in }|\D|\}\overset{\cong}\longrightarrow \{\text{\rm thick ideals in }\stab(\D)\},\vspace{1mm}\\
\Theta\mapsto \msc{K}_\Theta,
\end{array}
\]
where $\msc{K}_\Theta$ is the thick ideal of all objects $V$ in $\stab(\D)$ which are supported in the given set $|\D|_V\subset \Theta$.
\end{theoremA}

One can compare with analogous classification results for finite groups \cite{rickard97}, and finite group schemes \cite{friedlanderpevtsova07}.  By a thick \emph{prime} ideal in $\stab(\D)$ we mean a thick ideal $\msc{P}$ in $\stab(\D)$ which satisfies the following: a product $V\ot W$ is in $\msc{P}$ if and only if $V$ or $W$ is in $\msc{P}$.  Balmer has shown that the collection of prime ideals in $\stab(\D)$ admits the structure of a locally ringed space, which he calls the spectrum of $\stab(\D)$.
\par

Theorem \ref{thm:thick_id} implies the following calculation of the Balmer spectrum $\Spec(\stab(\D))$ for the Drinfeld double.

\begin{theoremA}[\ref{thm:spec}]
There is an isomorphism of locally ringed spaces
\[
f_{\rm coh}:|\D|\overset{\cong}\longrightarrow \Spec(\stab(\D)).
\]
\end{theoremA}

We note that the proofs of Theorems \ref{thm:thick_id} and \ref{thm:spec} rely on the construction of a certain ``hybrid", Benson-Iyengar-Krause-type support theory \cite{bensoniyengarkrause08} for infinite-dimensional $\D$-representations.  We discuss this support theory in Section \ref{sect:bigtensor} below.
\par

Let us provide, in closing, an elaboration on the methods employed in our analysis of the center $Z(\rep(G))=\rep(\D)$, and on a related $\pi$-point construction which appears in the appendix.

\subsection{Elaborations on methods}

Our proofs of the above results intertwine various approaches to support varieties in
the literature.  There are, however, some fundamental mechanism which we leverage throughout the text.
\par

Our basic approach to support for the double is as follows: We show in Section \ref{sect:Sigma} that, for $G$ a Frobenius kernel in a sufficiently nice algebraic group $\mbb{G}$, the representation category of the Drinfeld double $\D=D(G)$ admit an ``effective comparison" with the representation category of an associated infinitesimal group scheme $\Sigma$.  In particular, there is a linear abelian, \emph{non-}tensor, equivalence
\begin{equation}\label{eq:180}
\mcl{L}:\rep(\D)\overset{\sim}\to \rep(\Sigma)
\end{equation}
which nonetheless transports support theoretic information back and forth.  For example, we have an identification of cohomological supports $|\D|_V=|\Sigma|_{\mcl{L}(V)}$ for all $V$ in $\rep(\D)$ (see Lemma \ref{lem:751}).
\par

The fact that the equivalence $\mcl{L}$ identifies support for $\D$ with that of $\Sigma$ is not a casual one, and requires one to ``descend" the equivalence $\mcl{L}$ to a family of local Hopf subalgebras $\D_\psi\subset \D$ which ``covers" $\D$.  This family of subalgebras $\{\D_\psi\}_{\psi\in V_r(G)}$ is parametrized by the scheme $V_r(G)$ of $1$-parameter subgroups in $G$, and plays a fundamental role in our study.  As a basic point, one can use the subalgebras $\D_\psi$ to detect projectivity of $\D$-representations.  In particular, a given $\D$-representation is projective if and only if its restriction to each $\D_\psi$ is projective (Theorem \ref{thm:proj_check}).  The ability of the $\D_\psi$ to detection projectivity of $\D$-representations is the covering property referred to above.
\par

The effective comparison \eqref{eq:180} is integral to our proofs of the tensor product property $|\D|_{V\ot W}=|\D|_V\cap |\D|_W$, and also to the classification results listed above.  Additionally, the particular nature of our comparison indicates the existence of a $\pi$-point support theory for representations of the double, which we discuss in more detail below.
\par

One might compare our approach with Avrunin and Scott's proof of Carlson's conjecture, where a certain change of coproduct result is used to relate supports for abelian restricted Lie algebras to those of elementary abelian groups \cite{avruninscott82}.  Similar change of coproduct methods are employed in recent work of the first author as well \cite{friedlander}.

\subsection{Conceptualizations via $\pi$-points} The introduction of $\pi$-points by Pevtsova and the first author \cite{friedlanderpevtsova05,friedlanderpevtsova07} provide an alternate way to conceptualize our results. Our discussion of an analogous theory of $\pi$-points for the Drinfeld double $\D$ is relegated to the appendix because they do not figure directly into the proofs of the results we have summarized.  Instead, these results justify the intuition of $\pi$-points.
\par

For us, a $\pi$-point for $\D$ is a choice of field extension $K/k$, and a flat algebra map $\alpha:K[t]/(t^p)\to \D_K$ which admits an appropriate factorization through one of the local Hopf subalgebras $\D_\psi\subset \D_K$ (Definitions \ref{def:pi_psi} and \ref{def:1445}).  We then construct the space $\Pi(\D)$ of equivalence classes of $\pi$-points, and a corresponding $\pi$-point support theory $V\mapsto \Pi(\D)_V$ for the double.  The support spaces $\Pi(\D)_V$ are explicitly the locus of all $\pi$-points $\alpha$ at which the restriction $\res_\alpha(V_K)$ of $V$ to $K[t]/(t^p)$ is non-projective.  
\par

Two of our main results are that $\pi$-point support for the double $\D$ satisfies the tensor product property
\begin{equation}\label{eq:190}
\Pi(\D)_{V\ot W}=\Pi(\D)_V\cap \Pi(\D)_W
\end{equation}
(Theorem \ref{thm:tpp_Pi}), and also agrees with cohomological support.  In the statement of the following theorem we suppose that $G$ is, as usual, a Frobenius kernel in a sufficiently nice algebraic group $\mbb{G}$, i.e.\ one which admits a quasi-logarithm.

\begin{theoremA}[\ref{thm:Pi}]
Consider $G$ as above, and $\D=D(G)$.  There is a homeomorphism of topological spaces
\[
\Pi(\D)\overset{\cong}\longrightarrow |\D|
\]
which restricts to a homeomorphism of support spaces $\Pi(\D)_V\overset{\cong}\to |\D|_V$ for each $V$ in $\rep(\D)$.
\end{theoremA}

We furthermore construct a ``universal" $\pi$-point theory $\Pi_\ot(\D)_\star$, and show that our specific $\pi$-point support theory $\Pi(\D)_\star$ agrees with this universal theory.  One can see Theorem \ref{thm:Pi_ot} below.
\par

In considering the $\pi$-point perspective for support, we open up the \emph{possibility} of a deeper analysis of support for the double via explicit nilpotent operators.  One can compare with the introduction of local Jordan types for group representations in \cite{carlsonfriedlanderpevtsova08,friedlanderpevtsovasuslin07}, and constructions of vector bundles on support spaces provided in \cite{friedlanderpevtsova11,bensonpevtsova12}.  Although we won't discuss the issue here, our methods also allow us to identify cohomological and hypersurface supports for Drinfeld doubles of first Frobenius kernels $\mbb{G}_{(1)}$ in sufficiently nice algebraic groups (cf.\ \cite[Corollary 7.2, \S 13.3]{negronpevtsovaI}).

\subsection{Acknowledgments}

Thanks to Jon Carlson, Srikanth Iyengar, and Julia Pevtsova for helpful conversation.  The second author is supported by NSF grant DMS-2001608.

\tableofcontents

\section{Preliminaries}
\label{sect:prelim}

We recall basic information about Hopf algebras, finite group schemes, and the Drinfeld double construction.  We also recall the notion of cohomological support, and some basic results about Carlson modules.  Throughout this text we work over a base field $k$ which is of (finite) characteristic $p$.

\subsection{Hopf algebras and some generic notation}

We set some global notations, and recall a strong form of the Larson-Radford theorem \cite{larsonradford69}.  We assume the reader has some familiarity with Hopf algebras, and our canonical reference for the topic is Montgomery's text \cite{montgomery93}.

For us, a representation of a finite-dimensional algebra $A$ is the same thing as an $A$-module, and all representations/modules are \emph{left} representations/modules.  For a finite-dimensional Hopf algebra $A$ we let
\[
\rep(A):=\{\text{the tensor category of finite-dimensional $A$-representations}\}
\]
and
\[
\Rep(A):=\{\text{the monoidal category of all $A$-representation}\}.
\]
To be clear, when we say $\rep(A)$ is a \emph{tensor} category we recognize that all objects in $\rep(A)$ admit both left and right duals \cite[\S 2.10]{egno15}, whereas objects in $\Rep(A)$ are not dualizable in general.  We let $\operatorname{Irrep}(A)$ denote the collection of all (isoclasses of) simple $A$-representations.
\par

Throughout the text we denote finite-dimensional representations by the letters $V$ and $W$, and reserve the letters $M$ and $N$ for possibly infinite-dimensional representations.  This notation is employed throughout the text, without exception.
\par

We recall the following basic result, which will be of use later.

\begin{theorem}[\cite{larsonradford69}]
Any finite-dimensional Hopf algebra $A$ is Frobenius.  In particular, an $A$-representation $M$ is projective if and only if $M$ is injective.
\end{theorem}

\begin{proof}
The algebra $A$ is Frobenius by Larson and Radford \cite{larsonradford69}.  We note that if $A$ is Frobenius then injectivity is the same as projectivity, even for infinite-dimensional modules, by \cite[Theorem 5.3]{faithwalker67}.
\end{proof}

\subsection{Finite group scheme}

All group schemes in this text are affine.  A group scheme $G$, over a base field $k$, is called finite if it is finite as a scheme over $\Spec(k)$.  Rather, $G$ is finite if it is affine and has finite-dimensional (Hopf) algebra of global functions $\O(G)$.  For such finite $G$ we let $kG$ denote the associated group algebra $kG=\O(G)^\ast$.  A finite group scheme is called \emph{infinitesimal} if $G$ is connected, i.e.\ if $\O(G)$ is local, and \emph{unipotent} if the group algebra $kG$ is local.
\par

Following the framework of the previous section, we let $\rep(G)$ denote the category of finite-dimensional $kG$-modules, and $\Rep(G)$ denote the category of arbitrary $kG$-modules.  Note that $kG$-modules are identified with $\O(G)$-comodules as in \cite[Lemma 1.6.4]{montgomery93}, so that finite-dimensional $kG$-modules are in fact identified with $k$-linear representations of the group scheme $G$.

\subsection{The Drinfeld double and the Drinfeld center}
\label{sect:Z}

Let $G$ be a finite group scheme.  The adjoint action of $G$ on itself induces an action of $kG$ on $\O(G)$, and we can form the corresponding smash product, which is known as the \emph{Drinfeld double}, or \emph{quantum double} of $kG$, $D(G)=\O(G)\# kG$.  We usually employ the generic notation $\D$ for the Drinfeld double
\[
\D:=D(G).
\]
The algebra $\D$ admits a unique Hopf algebra structure for which the two algebra inclusions $\O(G)\to \D$ and $kG\to \D$ are inclusions of Hopf algebras.  See for example \cite[Corollary 10.3.10]{montgomery93}.

\begin{remark}
There is an analogous construction $A\rightsquigarrow D(A)$ of the Drinfeld double for an arbitrary finite-dimensional Hopf algebra $A$.  So, we are only discussing a particular instance of a general construction.
\end{remark}

\begin{remark}
If one compares directly with the presentation of \cite{montgomery93}, then one finds an alternate description of the double as a smash product between the coopposite Hopf algebra $\O(G)^{\rm cop}$ and $kG$.  However, by applying the antipode to the $\O(G)$ factor in $\D$, one sees that the cooposite comultiplication on $\O(G)$ can be replaced with the usual one, up to Hopf isomorphism.
\end{remark}

From a categorical perspective, we can consider the Drinfeld center of the representation category $\rep(G)$.  This is the category of pairs
\[
Z(\rep(G))=\left\{\begin{array}{c}
\text{pairs $(V,\gamma_V)$ of an object $V$ in $\rep(G)$, and }\\
\text{a choice of half braiding }\gamma_V:V\ot -\to -\ot V
\end{array}\right\}
\]
Such a half-braiding $\gamma_V$ is required to be a natural isomorphism of endofunctors of $\rep(G)$, and we require that this natural isomorphism satisfies the expected compatibilities with the tensor structure on $\rep(G)$ \cite[Definition XIII.4.1]{kassel12}.
\par

The center $Z(\rep(G))$ inherits a tensor structure from that of $\rep(G)$, and admits a canonical braiding $c_{V,W}:V\ot W\to W\ot V$ induced by the given half-braidings on objects $\gamma_{V,W}:V\ot W\to W\ot V$.  This braiding on $Z(\rep(G))$ is highly non-symmetric, in any sense which one might consider \cite{shimizu19}.  For example, any object $V$ in $Z(\rep(G))$ for which the square braiding is trivial $c^2_{V,-}=id_{V\ot-}$ must itself be trivial, $V\cong \1^{\oplus \dim(V)}$.
\par

We have the following categorical interpretation of the double.

\begin{theorem}[{\cite[Theorem XIII.5.1]{kassel12}}]\label{thm:146}
For any finite group scheme $G$, there is an equivalence of tensor categories $\rep(\D)\cong Z(\rep(G))$.
\end{theorem}

As a corollary to this result, we see that the category $\rep(\D)$ of representations for the Drinfeld double is canonically braided.  This point is relevant for many applications in mathematical physics, and is also relevant in studies of support and cohomology.  Specifically, many support theoretic results which are stated in the context of symmetric tensor categories can be immediately extended to the braided setting.

\begin{remark}
As with the construction of the Drinfeld double, one can construct the Drinfeld center of an arbitrary finite tensor category.  Furthermore, the obvious analog of Theorem \ref{thm:146} is valid when we replace $\rep(G)$ with the representation category of an arbitrary finite-dimensional Hopf algebra.
\end{remark}

In addition to considering the double $\D$ we also consider a certain class of Hopf subalgebras $\D'\subset \D$ which one associates to subgroups in $G$.  The following lemma will prove useful for our analysis of the subalgebras $\D'$.

\begin{lemma}\label{lem:76}
Suppose that $G$ is an infinitesimal group scheme, and let $H\subset G$ be a closed subgroup in $G$.  Let $H$ act on $\O(G)$ via the (restriction of the) adjoint action, and consider the smash product algebra $\O(G)\# kH$.
\par

Restriction along the surjective algebra map $\O(G)\# kH\to kH$, $f\ot x\mapsto \epsilon(f)x$, provides a bijection
\[
\operatorname{Irrep}(H)\overset{\cong}\longrightarrow \operatorname{Irrep}(\O(G)\# kH).
\]
\end{lemma}

\begin{proof}
Same as the proof of \cite[Proposition 5.5]{friedlandernegron18}.
\end{proof}

\subsection{Cohomological support}
\label{sect:cohom_supp}

\begin{definition}
We say a finite-dimensional Hopf algebra $A$ (over $k$) has finite type cohomology (over $k$) if the following two contions hold:
\begin{itemize}
\item[(a)] The extensions $\Ext^\ast_A(k,k)$ form a finitely generated $k$-algebra.
\item[(b)] For any pair of finite-dimensional $A$-representations $V$ and $W$, the extensions $\Ext^\ast_A(V,W)$ form a finitely generated module over $\Ext^\ast_A(k,k)$, via the tensor action
\[
-\ot -:\Ext^\ast_A(k,k)\ot\Ext^\ast_A(V,W)\to \Ext^\ast_A(V,W).
\]
\end{itemize}
\end{definition}

Let $A$ be a finite-dimensional Hopf algebra, and suppose that $A$ has finite type cohomology.  We take
\[
|A|:=\Proj\Ext^\ast_A(k,k).
\]
Formally, $\Proj\Ext^\ast_A(k,k)$ is the topological space of homogeneous prime ideals in $\Ext^\ast_A(k,k)$, which we equip with the Zariski topology.  Since $\Ext^\ast_A(k,k)$ is graded commutative and finitely generated, restriction along the inclusion $\Ext^{ev}_A(k,k)\to \Ext^\ast_A(k,k)$ provides a homeomorphism $\Proj\Ext^\ast_A(k,k)\cong \Proj\Ext^{ev}_A(k,k)$.  The structure sheaf on $\Proj\Ext^\ast_A(k,k)$ is the expected one, whose sections over a basic open $D_f$, $f\in \Ext^n_A(k,k)$, are the degree $0$ elements in the localization $\Ext^\ast_A(k,k)_f$.
\par

For any finite-dimensional $A$-representation $V$, we can consider the self-extensions $\Ext^\ast_A(V,V)$ and the tensor action of $\Ext^\ast_A(k,k)$ on these extensions.  Note that the extensions of $V$ form a \emph{graded} module over $\Ext^\ast_A(k,k)$, and we may consider the associated sheaf $\Ext^\ast_A(V,V)^\sim$ on $|A|=\Proj\Ext^\ast_A(k,k)$.  We define the cohomological support of $V$ as the support of its associated sheaf
\begin{equation}\label{eq:support}
|A|_V:=\operatorname{Supp}_{|A|}\Ext^\ast_A(V,V)^\sim.
\end{equation}

We have the following basic claim.

\begin{lemma}[{\cite[Proposition 2]{pevtsovawitherspoon09}}]
Suppose that $A$ has finite type cohomology.  A finite-dimensional $A$-representation $V$ is projective if and only if its support vanishes, $|A|_V=\emptyset$.
\end{lemma}

In considering the aforementioned collection of Hopf subalgebras $\mfk{D}'\subset \mfk{D}$ we also take account of the following.

\begin{lemma}\label{lem:131}
Suppose that $A$ has finite type cohomology, and that $B\to A$ is an inclusion of Hopf algebras.  Then
\begin{enumerate}
\item $B$ has finite type cohomology.
\item The restriction map $\Ext^\ast_A(k,k)\to \Ext^\ast_B(k,k)$ is a finite algebra map, and the induced map on spectra
\[
\res^\ast:\Spec\Ext^\ast_B(k,k)\to \Spec\Ext^\ast_A(k,k)
\]
is such that $(\res^\ast)^{-1}(0)=\{0\}$.
\end{enumerate}
\end{lemma}

\begin{proof}
The algebra $B$ has finite-type cohomology, and the algebra map of (2) is finite, by \cite[Proposition 3.3]{negronplavnik}.  Since this map is finite, the fiber
\[
k\ot_{\Ext^\ast_A(k,k)}\Ext^\ast_B(k,k)
\]
is a finite-dimensional non-negatively graded algebra, and hence the irrelevant ideal is the unique prime ideal in this algebra.  This implies that the preimage $(\res^\ast)^{-1}(0)$ is the singleton $\{0\}$.
\end{proof}

Lemma \ref{lem:131} (2) tells us that restriction $\res:\rep(A)\to \rep(B)$ induces a \emph{well-defined} map on projective spectra $|B|\to |A|$.  This map is furthermore closed and has finite fibers.

\subsection{Cohomological support for group schemes}

In considering a finite group scheme $G$ (over $k$) we adopt the particular notation
\[
|G|:=|kG|=\Proj \Ext_G^\ast(k,k).
\]
We may consider cohomological support for $G$-representations as described in Section \ref{sect:cohom_supp}.
\par

In addition to cohomological support, there are a number of additional support theories for $\rep(G)$ which one might employ in tandem.  In particular, when $G$ is an infinitesimal group scheme, one can consider the $k$-scheme $V_r(G)$ of $1$-parameter subgroups in $G$ and its associated support theory of \cite{suslinfriedlanderbendel97}.  Although we do not use this theory explicitly in the text, it does ``run in the background" of our analysis.  So we sketch a presentation of this support theory here.
\par

At fixed $r\geq 0$, $V_r(G)$ is the moduli space of group scheme maps $\mbb{G}_{a(r)}\to G$ \cite{suslinfriedlanderbendel97I}, and for any finite-dimensional $G$-representation $W$ one has an associated support space $V_r(G)_W$.  The support space $V_r(G)_W$ is specifically the non-projectivity locus of the representation $W$ in $V_r(G)$.  So, for example, a $k$-point $\alpha:\mbb{G}_{a(r)}\to G$ is in the support $V_r(G)_W$ precisely when the restriction $\res_\alpha(W)$ is a non-projective $\mbb{G}_{a(r)}$-representation.  The moduli space $V_r(G)$ is a conical scheme, and the supports $V_r(G)_W$ are closed conical subschemes in $V_r(G)$.  
\par

By results of \cite{suslinfriedlanderbendel97}, we have a natural scheme map $\Psi:\mbb{P}(V_r(G))\to |G|$ from the projectivization of $V_r(G)$, and this map is a \emph{homeomorphisms} whenever $G$ is of height $\leq r$.  The map $\Psi$ restricts to homeomorphisms $\Psi_W:\mbb{P}(V_r(G)_W)\to |G|_W$ between support spaces at arbitrary $W\in \rep(G)$, again when $G$ is of height $\leq r$.  So the support theory $V_r(G)_\star$ provides a kind of group theoretic ``realization" of cohomological support for infinitesimal group schemes.

\begin{remark}
Our notation $|G|$ conflicts slightly with the notation of \cite{suslinfriedlanderbendel97I,suslinfriedlanderbendel97,friedlanderpevtsova07}.  Namely, $|G|$ is used to denote the \emph{affine} spectrum of $\Ext^\ast_G(k,k)$ in the aforementioned papers, while we use it to denote the projective spectrum.
\end{remark}

\begin{remark}
By results of \cite{friedlanderpevtsova07}, the support theory $W\mapsto V_r(G)_W$ for $\rep(G)$ has a reasonable extension to the category $\Rep(G)$ of arbitrary $kG$-representation.
\end{remark}

\subsection{Carlson modules and support}

Consider a finite-dimensional Hopf algebra $A$ with finite type cohomology.  Define the $n$-th syzygy $\Omega^nk$ of the trivial representation via any choice of projective resolution of $k$, $0\to \Omega^nk\to P^{-(n-1)}\to\dots\to P^0\to k$.  Given an extension $\zeta\in \Ext^n_A(k,k)$, we can represent $\zeta$ as a map $\tilde{\zeta}:\Omega^nk\to k$ and define
\[
L_\zeta:=\ker\left(\tilde{\zeta}:\Omega^nk\to k\right).
\]
The object $L_\zeta$ is called a \emph{Carlson module} associated to $\zeta$.
\par

The object $L_\zeta$ is clearly not uniquely defined by $\zeta$, since the definition relies on a choice of representative for the map $\zeta:\Sigma^{-n}k\to k$ in the derived category $D^b(A)$.  However, $L_\zeta$ is unique up to isomorphism in the stable category for $A$, and so is sufficiently unique for most support theoretic applications.  Carlson modules have a number of exceedingly useful properties.  We recall a few of these properties here.

\begin{proposition}[{\cite[Proposition 3]{pevtsovawitherspoon09}}]\label{prop:carlson}
Consider arbitrary homogenous extension $\zeta\in \Ext^{ev}_A(k,k)$.  For any finite-dimensional $A$-representation $V$ there is an equality of supports
\begin{equation}\label{eq:157}
|A|_{(L_\zeta\ot V)}=Z(\zeta)\cap |A|_V.
\end{equation}
\end{proposition}

As a corollary to Proposition \ref{prop:carlson} we find

\begin{corollary}[{\cite[Corollary 1]{pevtsovawitherspoon09}}]\label{cor:174}
Any closed subset $\Theta$ in $|A|$ is realizable as the support of a product $L=L_{\zeta_1}\ot \dots\ot L_{\zeta_m}$ of Carlson modules, $\Theta=|A|_L$.
\end{corollary}

Carlson modules also enjoy certain naturality properties with respect to exact tensor functors.  We list a particular occurence of such naturality here.

\begin{lemma}\label{lem:159}
If $\iota:B\to A$ is an inclusion of Hopf algebras, and $L_\zeta$ is a Carlson module associated to an extension $\zeta\in \Ext^\ast_A(k,k)$ over $A$, then the restriction $\res_\iota(L_\zeta)$ is a Carlson module for the image of $\res_\iota(\zeta)\in \Ext^\ast_B(k,k)$ of this extension in $\Ext^\ast_B(k,k)$.
\end{lemma}

\begin{proof}
By the Nichols-Zoeller theorem \cite{larsonradford69}, $A$ is projective as a $B$-module.  So the result just follows from the fact that a projective resolution $P\to k$ of the unit over $A$ restricts to a projective resolution over $B$.
\end{proof}

\section{The Hopf subalgebras $\D_\psi$ and a projectivity test}
\label{sect:proj_test}

Let $G$ be an infinitesimal group scheme.  We show that the Drinfeld double $\D=D(G)$ admits a family of Hopf embeddings $\{\D_\psi\to \D\}_{\psi\in \text{1-param}}$ which is parametrized by the space of $1$-parameter subgroups in $G$.  Each of the Hopf algebras $\D_\psi$ is local, and so behaves like a ``unipotent subgroup" in $\D$.
\par

We show that the family $\{\D_\psi\to \D\}_{\psi\in\text{1-param}}$ can be used to check projectivity of arbitrary (possibly infinite-dimensional) $\D$-representations.  One can see Theorem \ref{thm:proj_check} below for a specific statement.  We furthermore show that the cohomological support $|\D|_V$ of a finite-dimensional $\D$-representation $V$ can be reconstructed from the support spaces $|\D_\psi|_{\res_\psi(V)}$ of the restrictions of $V$ to the various $\D_\psi$.
\par

The family of embeddings $\{\D_\psi\to \D\}_{\psi\in \text{1-param}}$ plays an integral role throughout our study, and is therefore a fundamental object of interest.  As implied above, an analysis of support for the double $\D$ will be shown to be reducible to an analysis of support for the local subalgebras $\D_\psi$.  One can compare with the group theoretic setting, where the support theory of a finite group scheme is similarly reducible to that of its unipotent subgroups (cf.\ \cite{friedlanderpevtsova05,friedlanderpevtsova07}).

\subsection{1-parameter subgroups}

Let $k$ be a field of characteristic $p>0$, and $G$ be an infinitesimal group scheme over $k$.  We let $G_K$ denote the base change along any given field extension $k\to K$.

\begin{definition}
An embedded $1$-parameter subgroup for $G$ is a pair $(K,\psi)$ of a field extension $k\to K$ and a closed map of group schemes $\psi:\mbb{G}_{a(s),K}\to G_K$.  We call $K$ the field of definition for such a $1$-parameter subgroup $\psi$.
\end{definition}

Of course, by $\mbb{G}_{a(r),K}$ we mean the base change of the $r$-th Frobenius kerel in $\mbb{G}_a$.  Let us take a moment to compare with \cite{suslinfriedlanderbendel97I,suslinfriedlanderbendel97}.
\par

In the texts \cite{suslinfriedlanderbendel97I,suslinfriedlanderbendel97}, by a $1$-parameter subgroup the authors mean an \emph{arbitrary} group map $\psi':\mbb{G}_{a(r),K}\to G_K$.  Having fixed a preferred quotient $\mbb{G}_{a(r)}\to \mbb{G}_{a(s)}$ for each $s\leq r$, such a group map specifies an integer $s\leq r$ and a unique factorization of $\psi'$ as a composition of the quotient $\mbb{G}_{a(r),K}\to \mbb{G}_{a(s),K}$ followed by an embedding $\psi:\mbb{G}_{a(s),K}\to G_K$.  In this way, the moduli space of $1$-parameter subgroups $V_r(G)$ employed in \cite{suslinfriedlanderbendel97} is identified with the moduli space of embedded $1$-parameter subgroups for $G$, provided $G$ is of height $\leq r$.  (One can define the moduli space of embedded $1$-parameter subgroups in precise analogy with \cite[Definition 1.1]{suslinfriedlanderbendel97I}.)  One thus translate freely between the language of \cite{suslinfriedlanderbendel97I,suslinfriedlanderbendel97} and the language we employ in this text.

Having clarified with this point, we recall the following essential results of Suslin-Friedlander-Bendel \cite[Proposition 7.6]{suslinfriedlanderbendel97} and Pevtsova \cite[Theorem 2.2]{pevtsova04}.

\begin{theorem}[{\cite{suslinfriedlanderbendel97,pevtsova04}}]\label{thm:pevtsova}
Consider an infinitesimal group scheme $G$.  An arbitrary $G$-representation $M$ is projective over $G$ if and only if for every field extension $k\to K$, and embedded $1$-parameter subgroup $\psi:\mbb{G}_{a(s),K}\to G_K$, the the base change $M_K$ is projective over $\mbb{G}_{a(s),K}$.
\end{theorem}

To be clear, when we say that $M_K$ is projective over $\mbb{G}_{a(s),K}$ we mean that $M_K$ restricts to a projective $\mbb{G}_{a(s),K}$-representation along the given map $\psi:\mbb{G}_{a(s),K}\to G_K$.
\par

When we consider a finite-dimensional representation $V$, and $k$ is algebraically closed, it suffices to check projectivity of $V$ after restricting to all $1$-parameter subgroups which are defined over $k$.

\begin{theorem}\cite{suslinfriedlanderbendel97}\label{thm:sfb}
Consider an infinitesimal group scheme $G$, and a finite-dimensional $G$-representation $V$.  Suppose also that the base field $k$ is algebraically closed.  Then $V$ is projective over $G$ if and only if, for every embedded $1$-parameter subgroup $\psi:\mbb{G}_{a(s)}\to G$ which is defined over $k$, $V$ is projective over $\mbb{G}_{a(s)}$.
\end{theorem}

\begin{proof}
Suppose that $V$ is projective when restricted to all such $\psi$.  Then \cite[Corollary 6.8]{suslinfriedlanderbendel97} tells us that $V$ has no closed points in its support.  Since the support $|G|_V$ is closed, we conclude that $|G|_V=\emptyset$, and hence that $V$ is projective.
\end{proof}

\begin{remark}
Since the category $\Rep(G)$ is Frobenius, we can replace projectivity with injectivity, or even flatness, in the statements of Theorem \ref{thm:pevtsova} and \ref{thm:sfb}.
\end{remark}

\subsection{A family of local subalgebras, and projectivity}

As we have just observed, $1$-parameter subgroups play an essential role in studies of support for infinitesimal group schemes.  We provide a corresponding family of Hopf subalgebras for the Drinfeld double.

\begin{definition}
Let $G$ be an infinitesimal group scheme, and $\psi:\mbb{G}_{a(s),K}\to G_K$ be an embedded $1$-parameter subgroup.  Let $\D=D(G)$ denote the Drinfeld double for $G$.  We define $\D_\psi$ to be the Hopf algebra
\[
\D_\psi:=\O(G_K)\# K\mbb{G}_{a(s),K},
\]
where $\mbb{G}_{a(s),K}$ acts on $\O(G_K)$ by restricting the adjoint action of $G_K$ along the given embedding $\psi$.
\end{definition}

Note that each Hopf algebra $\D_\psi$ embeds in the double $\D_K$ via the map $id_\O\ot \psi:\D_\psi\to \D_K$.  So we might speak of the $\D_\psi$ as Hopf subalgebras in $\D_K$, via a slight abuse of language.

\begin{lemma}
Consider an infinitesimal group scheme $G$, with Drinfeld double $\D$.  For any embedded $1$-parameter subgroup $\psi:\mbb{G}_{a(s)}\to G$ the Hopf algebra $\D_\psi$ is local.
\end{lemma}

\begin{proof}
By changing base if necessary we may assume $K=k$.  By Lemma \ref{lem:76} the restriction map provides an bijection $\operatorname{Irrep}(\mbb{G}_{a(s)})\to \operatorname{Irrep}(\D_\psi)$.  Now, since $\mbb{G}_{a(s)}$ is unipotent, the trivial representation is the only simple object in $\rep(\mbb{G}_{a(s)})$.  So we observe that $\rep(\D_\psi)$ has a unique simple object, and therefore that $\D_\psi$ is local.
\end{proof}

We recall that, according to Theorem \ref{thm:pevtsova}, $1$-parameter subgroups in a given infinitesimal group scheme can be used to detect projectivity of $G$-representations.  We observe an analogous detection property for the $\D_\psi$.

\begin{theorem}\label{thm:proj_check}
Consider an arbitrary representation $M$ over the Drinfeld double $\D$ of an infinitesimal group scheme $G$.  Then $M$ is projective over $\D$ if and only if for every field extension $k\to K$, and every embedded $1$-parameter subgroup $\psi:\mbb{G}_{a(s),K}\to G_K$, the base change $M_K$ is projective over $\D_\psi$.
\par

When $M$ is finite-dimensional, and $k$ is algebraically closed, $M$ is projective over $\D$ if and only if, for all embedded $1$-parameter subgroups $\psi:\mbb{G}_{a(s)}\to G$ which are defined over $k$, $M$ is projective over $\D_\psi$.
\end{theorem}

\begin{proof}
Recall that $\D$ is Frobenius, so that projectivity of $M$ is equivalent to injectivity.  It suffices to check projectivity/injectivity after changing base to the algebraic closure $\bar{k}$, so that we may assume $k=\bar{k}$.  Furthermore, as with any finite dimensional algebra, injectivity of $M$ is equivalent to vanishing of the extensions
\[
\Ext^{>0}_{\D}(S,M)=0\ \ \text{from the sum $S$ of all simple $\D$-reps.}
\]
So we seek to establish the above vanishing of cohomology.  In what follows we take $\O=\O(G)$.
\par

If $M$ is projective over $\D$, then $M$ is projective over the Hopf subalgebra $\O\subset \D$ \cite[Theorem 3.1.5]{montgomery93}.  Thus $M$ is injective over $\O$ in this case.  Similarly, if $M_K$ is projective over $\D_\psi$, then $M_K$ is projective over $\O_K$, and thus injective over $\O_K$ as well.  It follows that $M$ is injective over $\O$ itself.  So it suffices to assume that $M$ is injective over $\O$, and prove that in this case $M$ is injective over $\D$ if and only if $M_K$ is injective over $\D_\psi$ for all extensions $k\to K$ and embeddings $\psi:\mbb{G}_{a(s),K}\to G_K$.
\par

Let us assume that $M$ is injective over $\O$.  By Lemma \ref{lem:76}, all simple $\D$-representations are restrictions of simple $G$-representations along the projection $\D\to kG$.  It follows that we have a spectral sequence
\[
\Ext^\ast_G(S,\Ext^\ast_\O(k,M))\ \Rightarrow\ \Ext^\ast_{\D}(S,M)
\]
which reduces to an identification
\[
\Ext^\ast_G(S,\Hom_\O(k,M))=\Ext^\ast_{\D}(S,M),
\]
since $M$ is injective over $\O$.  Similarly, we have an identification
\[
\Ext^\ast_{\mbb{G}_{a(s),K}}(K,\Hom_{\O_K}(K,M_K))=\Ext^\ast_{\D_\psi}(K,M_K)
\]
at any embedded $1$-parameter subgroup $\psi:\mbb{G}_{a(s),K}\to G_K$.  Hence $M$ is injective over $\D$ (resp.\ $M_K$ is injective over $\D_\psi$) if and only if the invariant subspace $\Hom_\O(k,M)$ is injective over $G$ (resp.\ $\Hom_{\O_K}(K,M_K)$ is injective over $\mbb{G}_{a(s),K}$).
\par

Given the above information, we seek to establish the claim that
\[
\begin{array}{l}
\Hom_\O(k,M)\text{ is injective over }G\vspace{1mm}\\
\ \ \Leftrightarrow\ \ \text{for each }\psi:\mbb{G}_{a(s),K}\to G_K,\ \Hom_\O(k,M)_K=\Hom_{\O_K}(K,M_K)\\
\hspace{1cm}\text{is injective over }\mbb{G}_{a(s),K}.
\end{array}
\]
But this final claim follows by Theorem \ref{thm:pevtsova}.  Similarly, one refers to Theorem \ref{thm:sfb} in the case of finite-dimensional $M$ to obtain the desired result.
\end{proof}

\subsection{(Re)constructing cohomological support}

We consider cohomological support for \emph{finite-dimensional} representations over the Drinfeld double.  Fix an infinitesimal group scheme $G$ and let $\D$ denote its Drinfeld double $\D=D(G)$.  Recall our notation $|\D|$ for the projective spectrum of cohomology, $|\D|=\Proj\Ext^\ast_\D(k,k)$.  We have the following basic result of \cite{friedlandernegron18,negron}.

\begin{theorem}[\cite{friedlandernegron18,negron}]\label{thm:fn}
The Drinfeld double $\D$ has finite type cohomology.
\end{theorem}

We now apply Lemma \ref{lem:131} and Theorem \ref{thm:fn} to find 

\begin{corollary}\label{cor:452}
For any embedded $1$-parameter subgroup $\psi:\mbb{G}_{a(s),K}\to G_K$, the Hopf algebra $\D_\psi$ has finite type cohomology, and the induced map on projective spectra $\res^\ast_\psi:|\D_\psi|\to |\D_K|$ is a finite map of schemes.
\end{corollary}

Let us consider an arbitrary field extension $k\to K$.  We note that the natural map $K\ot \Ext^\ast_\D(k,k)\overset{\cong}\to \Ext^\ast_\D(K,K)$ is an isomorphism, and thus identifies the spectrum $|\D_K|$ with the base change $|\D|_K$.  For any embedded $1$-parameter subgroup $\psi:\mbb{G}_{a(s),K}\to G_K$, we therefore obtain a map of schemes
\begin{equation}\label{eq:f_psi}
f_\psi:|\D_\psi|\to |\D|
\end{equation}
given by composing the map $\res^\ast_\psi:|\D_\psi|\to |\D_K|$ induced by restriction with the projection $|\D_K|=|\D|_K\to |\D|$.
\par

We note that these $f_\psi$ are not closed morphisms in general.  This is simply because the projection $|\D|_K\to |\D|$ does not preserve closed points when the extension $k\to K$ is infinite.  On the other hand, we see that any point $x$ in $|\D|$ is represented by--or rather lifts to--a closed point in the base change $|\D|_{\bar{k}(x)}$.  So, by employing base change one is able to treat arbitrary points in the spectrum $|\D|$ as closed points, at least to a certain degree.  We record a little lemma.

\begin{lemma}\label{lem:297}
Consider any finite-dimensional $\D$-representation $V$.
\begin{enumerate}
\item For an arbitrary field extension $k\to K$, the support $|\D_K|_{V_K}$ of $V_K$ over $\D_K$ is precisely the preimage of $|\D|_V$ along the projection $|\D_K|\to |\D|$.  In particular, the composition $|\D_K|_{V_K}\subset |\D_K|\to |\D|$ is a surjection onto $|\D|_V$.
\item For any embedded $1$-parameter subgroup $\psi:\mbb{G}_{a(s),K}\to G$ the map $f_\psi$ restricts to a morphism between support spaces $|\D_\psi|_{V_K}\to |\D|_V$.  In particular, the image of $|\D_\psi|_{V_K}$ under $f_\psi$ is contained in $|\D|_V$.
\end{enumerate}
\end{lemma}

\begin{proof}
Statement (1) follows from the fact that (a) For any scheme $X$, the projection $X_K\to X$ along a field extension $k\to K$ is surjective and (b) for any map of schemes $f:Y\to X$, and coherent sheaf $\msc{F}$ on $X$, we have $\operatorname{Supp}(f^\ast \msc{F})=f^{-1}\operatorname{Supp}(\msc{F})$.  For (2) it suffices to prove the result in the case $K=k$, by (1).  We simply consider the diagram
\[
\xymatrix{
\Ext^\ast_\D(k,k)\ar[r]^{-\ot V}\ar[d]_{\res_\psi} & \Ext^\ast_\D(V,V)\ar[d]^{\res_\psi}\\
\Ext^\ast_{\D_\psi}(k,k)\ar[r]^{-\ot V} & \Ext^\ast_{\D_\psi}(V,V)
}
\]
induced by the restriction functors, and note that the supports $|\D|_V$ and $|\D_\psi|_V$ are the subvarieties associated to the respective kernels of the algebra maps $-\ot V$.
\end{proof}

We now observe that the support of $V$ over $\D$ can be reconstructed from the supports of $V$ over the $\D_\psi$, where we allow $\psi$ to vary along all $1$-parameter subgroups for $G$.

\begin{proposition}\label{prop:reconstruct}
Let $G$ be an infinitesimal group scheme and $\mfk{D}=D(G)$ be the associated Drinfeld double.  For any finite-dimensional $\D$-representation $V$ there is an equality
\begin{equation}\label{eq:327}
|\D|_V=\bigcup_{1\text{\rm-param subgroups}}f_\psi\left(|\D_\psi|_{V_K}\right).
\end{equation}
\end{proposition}

To be clear, the equality \eqref{eq:327} is an equality of \emph{sets}.  Indeed, the support of a representation is itself simply a closed subset in the space $|\D|$.  Also, the union \eqref{eq:327} is explicitly taken over the collection of all embedded $1$-parameter subgroups in $G$, each of which consists of a pair of a field extension $K/k$ and an embedding $\psi:\mbb{G}_{a(s),K}\to G_K$.

\begin{proof}
If the support $|\D|_V$ vanishes, i.e.\ if $V$ is projective over $\D$, then Theorem \ref{thm:proj_check} tells us that all of the supports $|\D_\psi|_{V_K}$ vanish as well.  So the claimed equality holds when the support $|\D|_V$ is empty.
\par

Let us assume now that $V$ is \emph{not} projective over $\D$, and hence that the support $|\D|_V$ is non-vanishing.  By considering base change, and Lemma \ref{lem:297}, we see that the equality \eqref{eq:327} can be obtained from the following claim:
\vspace{2mm}

\noindent Claim:\emph{ When $k$ is algebraically closed, and $x$ is a closed point in $|\D|_V$, there is a $1$-parameter subgroup $\psi:\mbb{G}_{a(s)}\to G$ such that $x$ is in the image $f_\psi(|\D_\psi|_V)$.}
\vspace{2mm}

\noindent Let us verify this claim.
\par

We suppose that $k=\bar{k}$ and consider a closed point $x$ in $|\D|_V$.  Let $L$ be a product of Carlson modules with $|\D|_L=\{x\}$.  Then $|\D|_{L\ot V}=\{x\}$ and for any $1$-parameter subgroup $\psi:\mbb{G}_{a(s)}\to G$ we have
\[
f_\psi\left(|\D_\psi|_{L\ot V}\right)=\left\{
\begin{array}{ll}
\{x\} & \text{if }x\in f_\psi(|\D_\psi|_V)\vspace{1mm}\\
\emptyset & \text{else}.
\end{array}\right.
\]
Indeed, the above formula follows from the fact that $|\D_\psi|_L=f_\psi^{-1}(x)$, by Lemma \ref{lem:159}, and the subsequent fact that
\[
|\D_\psi|_{L\ot V}= f_\psi^{-1}(x)\cap |\D_\psi|_V,
\]
by Proposition \ref{prop:carlson}.
\par

Recall that, by the projectivity test of Theorem \ref{thm:proj_check}, projectivity of the restriction of $L\ot V$ along each such $\psi$ would imply that $L\ot V$ is projective over $\D$.  Equivalently, vanishing of the supports $|\D_\psi|_{L\ot V}$ along all such $\psi$ would imply vanishing of the support $|\D|_{L\ot V}$.  Since we have chosen $L$ so that the latter space explicitly \emph{does not} vanish, we conclude that some support space $|\D_\psi|_{L\ot V}$ does not vanish.  Rather, $x\in f_\psi(|\D_\psi|_{L\ot V})$ for some $\psi$, and thus $x\in f_\psi(|\D_\psi|_V)$ for some $\psi$.  So we have proved the above Claim, and thus establish the identification \eqref{eq:327}.
\end{proof}

We remark, in closing, that one can prove analogs of the results of this section for arbitrary finite (rather than infinitesimal) group schemes.  One simply replaces the ``testing groups" $\mbb{G}_{a(s)}$ with a larger class of unipotent group schemes (cf.\ \cite{friedlanderpevtsova05}).

\section{Quasi-logarithms for group schemes}
\label{sect:qlog}

In this short aside we recall the notion of a quasi-logarithm for an affine group scheme.  As we recall below, ``most" familiar algebraic groups admit quasi-logarithms.  One can see Proposition \ref{prop:44} in particular.  As our study of support for Drinfeld doubles becomes more focused, we employ quasi-logarithms to gain some leverage on the \emph{algebra} structure of the double $\D=D(G)$.

\subsection{Quasi-logarithms}

\begin{definition}[\cite{kazhdanvarshavsky}]
Let $G$ be an affine group scheme with Lie algebra $\mfk{g}$.  We consider $\mfk{g}$ as an affine scheme $\mfk{g}=\Spec(\Sym(\mfk{g}^\ast))$.  A quasi-logarithm for $G$ is a map of schemes $l:G\to \mfk{g}$ which
\begin{enumerate}
\item[(a)] is equivariant for the adjoint actions,
\item[(b)] sends $1\in G$ to $\{0\}\in \mfk{g}$,
\item[(c)] induces the identity on tangent spaces $T_1l=id_\mfk{g}$.
\end{enumerate}
\end{definition}

Concretely, if we let $m\subset \O(G)$ denote the maximal ideal associated to the point $1\in G$, then a quasi-logarithm for $G$ is a choice of ad-equivariant splitting $\mfk{g}^\ast\to m$ of the projection $m\to m/m^2=\mfk{g}^\ast$.  We note that, when $G$ is smooth over the base field $k$, such a quasi-logarithm induces an isomorphism on the respective formal neighborhoods $\hat{l}:\hat{G}_1\overset{\cong}\to \hat{\mfk{g}}_0$.  Also, when $G$ is infinitesimal any quasi-logarithm is a closed embedding.
\par

The following lemma is straightforward.

\begin{lemma}
Suppose a group scheme $\mbb{G}$ admits a quasi-logarithm $l:\mbb{G}\to \mfk{g}$.  Then for any positive integer $r$, the restriction $l|_{\mbb{G}_{(r)}}:\mbb{G}_{(r)}\to \mfk{g}$ provides a quasi-logarithm for the Frobenius kernel $\mbb{G}_{(r)}$.
\end{lemma}

Through the remainder of the text we often adopt the following hypotheses: We assume $\mbb{G}$ is a smooth algebraic group which admits a quasi-logarithm, then consider the Frobenius kernels $G=\mbb{G}_{(r)}$ at arbitrary $r>0$.  The previous lemma tells us that all such $G$ naturally inherit quasi-logarithms from any choice of quasi-logarithm for the ambient group $\mbb{G}$.  So in this way we obtain various families of infinitesimal group schemes which admit quasi-logarithms.

\subsection{Appearances of quasi-logs in nature}

We discuss the ``generic" presence of quasi-logarithms among reductive algebraic groups.  Let $\mbf{G}$ be an affine algebraic group which is defined over a localization $R=\mbb{Z}[1/n]$ of the integers, and suppose that $\mbf{G}$ is generically reductive.  That is to say, suppose that the rational form $\mbf{G}_\mbb{Q}$ is reductive.  Take $\O=\O(\mbf{G})$.
\par

Let $m\subset \O$ be the ideal associated to the identity $1\in \mbf{G}(R)$, and consider the coadjoint representation $\mfk{g}^\ast=m/m^2$.  The surjection $m\to \mfk{g}^\ast$ admits an ad-equivariant splitting $\mfk{g}^\ast_\mbb{Q}\to m_\mbb{Q}\subset \O_\mbb{Q}$ over the rationals, since $\mbf{G}_\mbb{Q}$ has semisimple representation theory \cite[Theorem 22.42]{milne17}.  This splitting is defined over a further localization $R'=\mbb{Z}[1/N]$, so that we obtain a quasi-logarithm $\mbf{G}_{R'}\to \mfk{g}_{R'}$ defined over $R'$.  It follows that for any field $k$ of characteristic $p$ which does not divide $N$, the group $\mbb{G}=\mbf{G}_k$ admits a quasi-logarithm.  We record this observation.

\begin{proposition}
Let $\mbf{G}$ be a algebraic group which is defined over a localization $R=\mbb{Z}[1/n]$ of the integers, and suppose that $\mbf{G}$ is generically reductive.  Then for any field $k$, in all but finitely many characteristics, the $k$-form $\mbb{G}=\mbf{G}_k$ admits a quasi-logarithm. 
\end{proposition}

If we consider split semisimple algebraic groups, for example, we can be much more precise about the characteristics at which our group $\mbb{G}=\mbf{G}_k$ admits a quasi-logarithm.  We can also deduce quasi-logarithms for various classes of algebraic groups which are related to such semisimple $\mbb{G}$.

\begin{proposition}[{\cite[\S 6.1]{friedlandernegron18}}]\label{prop:44}
The following algebraic groups admit a quasi-logarithm:
\begin{itemize}
\item The general linear group $\operatorname{GL}_n$, over any field in any characteristic.
\item Any split simple algebraic group in very good characteristic (relative to the corresponding Dynkin type).
\item Any Borel subgroup inside a split simple algebraic group, in very good characteristic.
\item The unipotent radical in such a Borel, in sufficiently large characteristic.
\end{itemize}
\end{proposition}

\section{The Drinfeld double $\D$ via an infinitesimal group scheme}
\label{sect:Sigma}

Let $\mbb{G}$ be a smooth algebraic group over $k$ which admits a quasi-logarithm, and let $G$ be a Frobenius kernel in $\mbb{G}$.  We consider the Drinfeld double $\D$ for $G$.  In this section we show that, for $G$ as described, there is a \emph{linear abelian} equivalence
\[
\mcl{L}:\rep(\D)\overset{\sim}\to \rep(\Sigma)
\]
between the representation category of the double and the representation category of an associated infinitesimal group scheme $\Sigma$.  We show, furthermore, that this equivalence restricts to a corresponding abelian equivalence $\mcl{L}_\psi:\rep(\D_\psi)\overset{\sim}\to \rep(\Sigma_\psi)$ at all embedded $1$-parameter subgroups in $G$.

Although these equivalences are not equivalences of tensor categories, they can be used in highly nontrivial ways in an analysis of support for the double, as we will see in Sections \ref{sect:tensor_fd} and \ref{sect:bigtensor}.

\subsection{The group schemes $\Sigma_V(G,r)$}

Consider a finite group scheme $G$ and any finite-dimensional $G$-representation $V$.  To $V$ we associate the algebra
\[
S_r(V):=\Sym(V)/(v^{p^r}:v\in V).
\]
This algebra has the natural structure of a cocommutative Hopf algebra in the symmetric tensor category $\rep(G)$, where the coproduct on $S_r(V)$ is defined by taking all of the generators $v\in V$ to be primitive $\Delta(v)=v\ot 1+1\ot v$ (cf.\ \cite[\S 1.3]{andruskiewitschschneider}).  Indeed, we may view $V$ as an abelian Lie algebra in $\rep(G)$, and consider the universal enveloping algebra $U(V)=\Sym(V)$.  We then obtain $S_r(V)$ as the quotient of $U(V)$ by the Hopf ideal generated by the primitive elements $v^{p^r}$, $v\in V$.
\par

Now, since the forgetful functor $\rep(G)\to Vect$ is a map of symmetric tensor categories, any Hopf algebra in $\rep(G)$ can be viewed immediately as a Hopf algebra in the classical sense, i.e.\ as a Hopf algebra in $Vect$.  So we may view $S_r(V)$ as a Hopf algebra in $\rep(G)$ or as a Hopf algebra in $Vect$ as needed.  Furthermore, for any Hopf algebra $S$ in $\rep(G)$ the smash product $S\# kG$ admits a unique Hopf algebra structure (in $Vect$) so that the two inclusions
\[
S\to S\# kG\ \ \text{and}\ \ kG\to S\# kG
\]
are maps of Hopf algebras (in $Vect$).  Indeed, this is the standard bosonization procedure \cite[Theorem 1.6.9]{radford}.  So, in the case discussed above, we obtain the following.

\begin{lemma}
For any finite group scheme $G$ and any finite-dimensional $G$-representation $V$, the smash product $S_r(V)\# kG$ admits a unique cocommutative Hopf algebra structure (in $Vect$) such that the following conditions hold:
\begin{enumerate}
\item[(a)] Each $v\in V$ is primitive.
\item[(b)] The inclusion $kG\to S_r(V)\# kG$ is a map of Hopf algebras.
\end{enumerate}
\end{lemma}

\begin{proof}
The existence of such a Hopf structure follows by the discussion above.  Cocommutativity follows from the fact that the two Hopf subalgebras $S_r(V)$ and $kG$ are cocommutative, and that the multiplication map
\[
mult:S_r(V)\ot kG\to S_r(V)\# kG
\]
is a morphism, and hence an isomorphism, of coalgebras.
\end{proof}

The fact that $S_r(V)\# kG$ is cocommutative tells us that it serves as the group ring for an associated finite group scheme.

\begin{definition}\label{def:A}
For any finite group scheme $G$, and any finite-dimensional $G$-representation $V$, we define the finite group scheme $\Sigma_V(G,r)$ to be the unique such group scheme with associated group algebra
\[
k\Sigma_V(G,r)=S_r(V)\# kG.
\]
\end{definition}

Said another way, $\Sigma_V(G,r)$ is the spectrum of the dual Hopf algebra
\[
\Sigma_V(G,r)=\Spec\left((S_r(V)\# kG)^\ast\right).
\]
Note that the group scheme $\Sigma_V(G,r)$ admits a normal subgroup $N_V(r)\subset \Sigma_V(G,r)$ which coresponds to the normal Hopf subalgebra $S_r(V)\subset k\Sigma_V(G,r)$, and that we have an exact sequence of group schemes
\begin{equation}\label{eq:458}
1\to N_V(r)\to \Sigma_V(G,r)\to G\to 1.
\end{equation}

\begin{lemma}
Suppose that $G$ is infinitesimal, and let $V$ be an any finite-dimensional $G$-representation.  Then $\Sigma_V(G,r)$ is infinitesimal.  Furthermore, if $G$ is unipotent then $\Sigma_V(G,r)$ is unipotent as well.
\end{lemma}

\begin{proof}
Take $\Sigma=\Sigma_V(G,r)$.  As a coalgebra $k\Sigma=S_r(V)\ot kG$.  So the algebra of functions $\O(\Sigma)$ is the tensor product $S_r(V)^\ast\ot \O(G)$.  Since $S_r(V)$ is a connected coalgebra, with primitive space $\operatorname{Prim}(S_r(V))=\{v^{p^s}:0\leq s<r\}$, it follows that the dual $S_r(V)^\ast$ is local.  Since $G$ is infinitesimal the algebra $\O(G)$ is also local.  Now, since a tensor product of finite-dimensional local $k$-augmented algebras is also local, we see that $\O(\Sigma)$ is local.  Hence $\Sigma$ is infinitesimal.
\par

For arbitrary $G$, the maximal ideal $m=(V)\subset S_r(V)$ is stable under the action of $G$, so that the ideal $m\ot kG\subset k\Sigma$ is nilpotent.  Hence the Jacobson radical of $k\Sigma$ is the preimage of the Jacobson radical in $kG$ along the surjection $k\Sigma\to kG$.  It follows that if $kG$ is local then $k\Sigma$ is local.  So we see that $\Sigma$ is unipotent when $G$ is unipotent.
\end{proof}

We note, finally, that the group scheme $\Sigma_V(G,r)$ can be defined entirely within the category of group schemes (rather than in the category of Hopf algebras).  Indeed, the action of $G$ on $V$ induces an action on the $r$-th Frobenius kernel in corresponding additive group scheme $V_a=(V,+)$, and hence on the Cartier dual $(N_V(r)=)V^\vee_{a(r)}$.  We then recover $\Sigma_V(G,r)$ as the semidirect product $V^\vee_{a(r)}\rtimes G$.  This construction is more in line with the standard perspective of, say, Jantzen's text \cite{jantzen03}.  However, what is of interest to us is the algebra structure on $k\Sigma_V(G,r)$.  So the above Hopf algebraic presentation is sufficiently informative for our purposes.

\subsection{Quasi-logarithms and a system of linear equivalences}
\label{sect:soe}

We consider the above construction $\Sigma_V(G,r)$ for the coadjoint representation of $G$.

\begin{definition}\label{def:Sigma}
For any finite group scheme $G$ we define
\[
\Sigma(G,r):=\Sigma_{\mfk{g}^\ast}(G,r),
\]
where $\mfk{g}^\ast$ is the coadjoint representation.  Additionally, for any embedded $1$-parameter subgroup $\psi:\mbb{G}_{a(s),K}\to G_K$ we restrict the coadjoint representation of $G_K$ along $\psi$ to define
\[
\Sigma(G,r)_\psi:=\Sigma_{\mfk{g}^\ast_K}(\mbb{G}_{a(s),K},r).
\]
\end{definition}

When no confusion will arise we will be even more casual in our presentation, and write simply
\[
\Sigma=\Sigma(G,r),\ \ \Sigma_\psi=\Sigma(G,r)_\psi.
\]
(We usually consider a Frobenius kernel $G=\mbb{G}_{(r)}$ and the associated group scheme $\Sigma(G,r)$, so that the parameter $r$ is already clear from the context.)  Note that for any embedded $1$-parameter subgroup $\psi:\mbb{G}_{a(r),K}\to G_K$ the product map $id_{S_r}\ot K\psi$ provides a natural inclusion of group schemes $\Sigma(G,r)_\psi\to \Sigma(G,r)_K$.

\begin{lemma}\label{lem:536}
Let $\mbb{G}$ be a smooth algebraic group which admits a quasi-logarithm.  Consider $G=\mbb{G}_{(r)}$, $\mfk{D}=D(G)$, and $\Sigma=\Sigma(G,r)$ at arbitrary $r>0$.
\par

Any choice of quasi-logarithm $l$ for $G$ specifies an isomorphism of augmented $k$-algebras $a(l):k\Sigma\to\mfk{D}$.  Furthermore, for any $1$-parameter subgroup $\psi:\mbb{G}_{a(s),K}\to G_K$, we have a corresponding isomorphism of augmented $K$-algebra $a(l)_\psi:K\Sigma_\psi\to \mfk{D}_\psi$.  These isomorphisms fit into a diagram of algebra maps
\begin{equation}\label{eq:soi}
\xymatrix{
K\Sigma_K\ar[rr]^{a(L)_K} & & \D_K\\
K\Sigma_\psi\ar[u]^{\rm incl}\ar[rr]^{a(L)_\psi} & & \D_\psi.\ar[u]_{\rm incl}
}
\end{equation}
\end{lemma}

The augmentations considered above are, of course, the augmentations specified by the respective counits.

\begin{proof}
Take $S=S_r(\mfk{g}^\ast)$, with its $G$-action induced by the coadjoint action on $\mfk{g}^\ast$.  Any quasi-logarithm $l$ specifies a $G$-equivariant map of algebras $a_0:S\to \O(G)$ which is an isomorphism on cotangent spaces $m_0/m_0^2\to m_1/m_1^2$.  Indeed, a quasi-logarithm for $G$ is a choice of equivariant section $\mfk{g}^\ast\to m_1$ of the reduction map $m_1\to m_1/m_1^2=\mfk{g}^\ast$, and $a_0$ is the algebra map from the (truncated) symmetric algebra induced by this section.  Since $\O(G)$ is local, such a map is necessarily surjective.  Since furthermore $\dim(S)=\dim(\O(G))=r^{\dim(\mfk{g})}$, it follows that $a_0$ is an isomorphism.  Since both algebras in question are local, $a_0$ is an isomorphism of augmented algebras.  (This point is also obvious from the construction of $a_0$.)
\par

We obtain the desired isomorphism $a(l):k\Sigma\to \D$ as the product $a(l)=a_0\ot id_{kG}$, and similarly $a(l)_\psi:K\Sigma_\psi\to \D_\psi$ is the product $(a_0)_K\ot id_{k\mbb{G}_{a(s)}}$.  One sees directly that, since $a_0$ is an isomorphism of augmented algebras, $a(l)$ and $a(l)_\psi$ are also isomorphisms of augmented algebras.
\end{proof}

As a consequence of the above lemma, we see that any choice of quasi-logarithm for the ambient group $\mbb{G}$ specifies a ``system of linear equivalences" for $\D$, and its local family of Hopf subalgebras $\D_\psi$.

\begin{proposition}\label{prop:809}
For $G$ as in Lemma \ref{lem:536}, there is an equivalence of $k$-linear, abelian categories $\mcl{L}:\rep(\mfk{D})\overset{\sim}\to \rep(\Sigma)$ which preserves the trivial representation $\mcl{L}(k)=k$.  Furthermore, for any $1$-parameter subgroup $\psi:\mbb{G}_{a(s),K}\to G_K$ we have a corresponding equivalence of $K$-linear categories $\mcl{L}_\psi:\rep(\D_\psi)\overset{\sim}\to \rep(\Sigma_\psi)$ which preserves the trivial representation, and fits into a diagram of exact linear functors
\begin{equation}\label{eq:soe}
\xymatrix{
\rep(\D_K)\ar[rr]^{\mcl{L}_K}\ar[d]_{\res} & & \rep(\Sigma_K)\ar[d]^{\res}\\
\rep(\D_\psi)\ar[rr]^{\mcl{L}_\psi} & & \rep(\Sigma_\psi).
}
\end{equation}
\end{proposition}

\begin{proof}
Define $\mcl{L}$ and $\mcl{L}_\psi$ as restriction along the algebra isomorphisms $a(l)$ and $a(l)_\psi$ of Lemma \ref{lem:536}, respectively.
\end{proof}

For any $1$-parameter subgroup $\psi:\mbb{G}_{a(s),K}\to G_K$ we let
\[
f'_\psi:|\Sigma_\psi|\to |\Sigma|
\]
denote the corresponding map on projective spectra of cohomology.  Specifically, we consider the composite
\[
f'_\psi:=\left(\ |\Sigma_\psi|\overset{\res^\ast}\longrightarrow |\Sigma_K|=|\Sigma|_K\to |\Sigma|\ \right).
\]
Proposition \ref{prop:809} tells us that, at any $1$-parameter subgroup $\psi:\mbb{G}_{a(s),K}\to G$, the maps $f'_\psi$ fit into a diagram
\begin{equation}\label{eq:soe_sch}
\xymatrix{
|\D_\psi|\ar[rrr]^{f_\psi} & & & |\D|\\
|\Sigma_\psi|\ar[u]_{\mcl{L}_\psi^\ast}^\cong\ar[rrr]^{f'_\psi} & & & |\Sigma|\ar[u]_{\mcl{L}^\ast}^\cong
}
\end{equation}
of maps of $k$-schemes, where $f_\psi$ is as in \eqref{eq:f_psi}.
\par

Now, from \cite[Corollary 5.4.1]{suslinfriedlanderbendel97} we understand that any closed embedding $\Sigma_0\to \Sigma_1$ of group schemes induces a map on projective spectra of cohomology $|\Sigma_0|\to |\Sigma_1|$ which is \emph{universally injective}.  The universal modifier here simply indicates that each base change $|\Sigma_0|_K\to |\Sigma|_K$ is also injective.  So the above diagram \eqref{eq:soe_sch} implies the following basic result.

\begin{proposition}\label{prop:incl}
Consider a smooth algebraic group $\mbb{G}$, and take $G=\mbb{G}_{(r)}$.  Suppose that $\mbb{G}$ admits a quasi-logarithm.  Let $\psi:\mbb{G}_{a(s)}\to G$ be an embedded $1$-parameter subgroup which is defined over $k$.  Then the induced map on projective spectra of cohomology
\[
f_\psi:|\D_\psi|\to |\D|
\]
is universally injective.
\end{proposition}

The system of equivalences \eqref{eq:soe}, which we view as a family of equivalences parametrized by the space of $1$-parameter subgroups in $G$, can be leveraged in quite substantive ways in an analysis of support for the double $\D$.  Indeed, the following two sections essentially argue this point in both the finite-dimensional and infinite-dimensional context.

\section{Support and tensor products for finite-dimensional representations}
\label{sect:tensor_fd}

As in the previous section, we consider a Frobenius kernel $G$ in a smooth algebraic group $\mbb{G}$ which admits a quasi-logarithm.  We prove that cohomological support for the Drinfeld double $\D=D(G)$ satisfies the tensor product property
\begin{equation}\label{eq:881}
|\D|_{(V\ot W)}=|\D|_V\cap |\D|_W.
\end{equation}
Here $V$ and $W$ are specifically \emph{finite-dimensional} representations over $\D$.  This result appears in Theorem \ref{thm:tensor} below.  Our proof of Theorem \ref{thm:tensor} relies on an analysis of cohomological support, and the tensor product property, for representations over the local family $\D_\psi$.
\par

For any given $\D_\psi$ we argue that the behaviors of cohomological support are, essentially, independent of the choice of coproduct.  We elaborate on this point in Subsections \ref{sect:loc1} and \ref{sect:lil_classify} below.
\par

In Section \ref{sect:bigtensor}, we provide an extension of cohomological support, and of the identity \eqref{eq:881}, to the big representation category $\Rep(\D)$.  Such an extension allows us to apply methods of Rickard \cite{rickard97} to show that cohomological support can also be used to classify thick tensor ideals in the stable representation category for $\D$.

\subsection{Comparison with the $\pi$-point support of Appendix \ref{sect:pi}}
\label{sect:compare}

Before we begin, let us make a few points of comparison between the material of this section and the material of Appendix \ref{sect:pi}, for the $\pi$-point orientated reader.  In the appendix we produce a $\pi$-point support theory for the double $\D$, essentially by restricting to the local subalgebras $\D_\psi$ and considering such a theory for $\D_\psi$.
\par

We note that the proof of the tensor product property for cohomological support is, arguably, more difficult than the proof for $\pi$-point support (Theorem \ref{thm:tpp_Pi} below).  However, the proof that $\pi$-support \emph{agrees} with cohomological support uses precisely the same technology which is used in the proof of the tensor product property for cohomological support.  So, depending on one's inclinations, one may view Theorem \ref{thm:tensor} below essentially \emph{as} the claim that $\pi$-point support and cohomological support agree for Drinfeld doubles of the prescribed form.

\subsection{Supports and thick ideals for local Hopf algebras}
\label{sect:loc1}

Let $A$ be a finite-dimensional, \emph{local}, Hopf algebra.  Suppose additionally that $A$ has finite type cohomology.
\par

For $A$ as prescribed, the support \eqref{eq:support} of a given finite-dimensional representation $V$ can be computed as the support of the sheaf associated to the $\Ext^\ast_A(k,k)$-module $\Ext^\ast_A(k,V)$, where we act via the first coordinate
\begin{equation}\label{eq:606}
|A|_V=\operatorname{Supp}_{|A|}\Ext^\ast_A(k,V)^\sim.
\end{equation}
See for example \cite[Proposition 5.7.1]{benson98} or \cite[Proposition 2]{pevtsovawitherspoon09}.  That is to say, the support spaces $|A|_V$ do not depend on the choice of Hopf structure on $A$.
\par

Let us write $D^b(A)$ for the bounded derived category of finite-dimensional $A$-representations.  Recall that a thick subcategory in $D^b(A)$ is a full triangulated subcategory which is closed under taking summands, and a thick ideal in $D^b(A)$ is a thick subcategory which is additionally closed under the (left and right) actions of $D^b(A)$ on itself.  The following lemma is strongly related to the above identification \eqref{eq:606}.

\begin{lemma}\label{lem:614}
Consider a finite-dimensional local Hopf algebra $A$ which has finite type cohomology.  Any thick subcategory in $D^b(A)$ is stable under the tensor action of $D^b(A)$ on itself.  That is to say, the collection of thick ideals in $D^b(A)$ is identified with the collection of thick subcategories in $D^b(A)$.
\end{lemma}

\begin{proof}
Locality tells us that any complex $V$ in $D^b(A)$ is obtainable from the trivial representation via a finite sequence of extensions.  It follows that for any object $W$ in $D^b(A)$, the product $V\ot W$ is obtainable from $W=k\ot W$ via a finite sequence of extensions.  Hence $V\ot W$ is contained in the thick subcategory generated by $W$, for arbitrary $V$ and $W$ in $D^b(A)$.  Similarly, $W\ot V$ is contained in the thick ideal generated by $W$.
\par

Now, let $\msc{K}\subset D^b(A)$ be any thick subcategory.  By the above discussion we have $V\ot \msc{K}\subset \msc{K}$ and $\msc{K}\ot V\subset \msc{K}$ for all $V$ in $D^b(A)$.  This shows that $\msc{K}$ is a thick ideal.  Hence the inclusion
\[
\{\text{thick ideals in }D^b(A)\}\to \{\text{thick subcategories in }D^b(A)\}
\]
is an equality.
\end{proof}

We note that the definition of support \eqref{eq:support} works perfectly well for arbitrary objects in the bounded derived category.  Furthermore, when $A$ is local the expression \eqref{eq:606} remains valid for any $V$ in $D^b(A)$.  
\par

For an exact triangle $V\to W\to V'$ in $D^b(A)$, the long exact sequence in cohomology provides an exact sequence of $\Ext^\ast_A(k,k)$-modules
\[
\Ext^\ast_A(k,V)\to \Ext^\ast_A(k,W)\to \Ext^\ast_A(k,V').
\]
So there is an inclusion of supports $|A|_W\subset \left(|A|_V\cup|A|_{V'}\right)$ whenever we have such a triangle.  Additionally, for any sum $V=V_1\oplus V_2$ in $D^b(A)$ we have an equality $|A|_V=|A|_{V_1}\cup |A|_{V_2}$.  From these observations we deduce an inclusion
\[
|A|_W\subset |A|_V\ \ \text{whenever $W$ is in the thick subcat generated by }V.
\]

\begin{lemma}\label{lem:incl}
Consider a finite-dimensional local Hopf algebra $A$.  For any $V$ and $W$ in $D^b(A)$ there is an inclusion
\[
|A|_{(V\ot W)}\subset \left(|A|_V\cap |A|_W\right).
\]
\end{lemma}

\begin{proof}
The object $V\ot W$ is in the thick ideal generated by $V$, and hence the thick subcategory generated by $V$ by Lemma \eqref{lem:614}.  So $|A|_{(V\ot W)}\subset |A|_V$ by the above reasoning.  We similarly find $|A|_{(V\ot W)}\subset |A|_W$, which gives the claimed inclusion $|A|_{(V\ot W)}\subset |A|_V\cap |A|_W$.
\end{proof}

We note that the inclusion of Lemma \ref{lem:incl} does \emph{not} hold for an arbitrary Hopf algebra $A$.  One can see for example \cite{bensonwitherspoon14}.

\begin{remark}
The familiar reader is free to replace the derived category $D^b(A)$ with the stable category $\operatorname{stab}(A)$ in the above discussion.
\end{remark}

\subsection{Classification of thick ideals for local algebras}
\label{sect:lil_classify}

\begin{definition}\label{def:lil_classify}
Let $A$ be a finite-dimensional Hopf algebra which has finite type cohomology.  We say that cohomological support for $A$ \emph{classifies thick ideals} in $D^b(A)$ if an inclusion of supports $|A|_W\subset |A|_V$, for nonzero $W$ and $V$ in $D^b(A)$, implies that $W$ is in the thick ideal generated by $V$ in $D^b(A)$.
\end{definition}

The supposition that $W$ and $V$ are nonzero (non-acyclic) is necessary to avoid issues with perfect complexes.  Namely, any perfect complex has vanishing support, and yet the ideal of perfect complexes in $D^b(A)$ is not contained in the ideal of acyclic complexes.  However, for nonzero $V$, we always have that $\operatorname{perf}(A)$ is contained in the thick ideal generated by $V$.
\par

One can consider representation categories of finite group schemes, for example.  In this case we understand \cite{friedlanderpevtsova07} that cohomological support does in fact classify thick ideals in the associated derived category.

\begin{theorem}[{\cite[Theorem 6.3]{friedlanderpevtsova07}}]\label{thm:G_class}
For any finite group scheme $G$, cohomological support classifies thick ideals in $D^b(G)$.
\end{theorem}

When $G$ is furthermore \emph{unipotent}, or rather when $\rep(G)$ is a local category, Theorem \ref{thm:G_class} and Lemma \ref{lem:614} combine to give the following.

\begin{corollary}\label{cor:U_class}
Suppose that $G$ is a finite unipotent group scheme.  Then thick \emph{subcategories} in $D^b(G)$ are classified by cohomological support.
\end{corollary}

The following will prove quite useful in our analysis of support for the local Hopf algebras $\D_\psi$.

\begin{proposition}\label{prop:tpp_A}
Let $A$ be a finite-dimensional local algebra.  Suppose that $A$ admits a Hopf algebra structure for which cohomological support classifies thick ideals in the derived category $D^b(A)$.  Then under \emph{any} choice of Hopf structure on $A$, and any choice of objects $V$ and $W$ in $D^b(A)$, we have an equality
\[
|A|_{(V\ot W)}=|A|_V\cap |A|_W.
\]
\end{proposition}

\begin{proof}
Let $\langle X\rangle$ denote the thick subcategory generated by a given object $X$ in $D^b(A)$.  For any object $L$ in $\langle V\rangle$ the product $L\ot W$ is in $\langle V\ot W\rangle$, and hence $|A|_{(L\ot W)}\subset |A|_{(V\ot W)}$.  Consider a product of Carlson modules $L$ for which $|A|_L=|A|_V$.  Then by Proposition \ref{prop:carlson} we have
\[
|A|_{(V\ot W)}\supset |A|_{(L\ot W)}=|A|_L\cap|A|_W=|A|_V\cap |A|_W.
\]
The opposite inclusion is covered by Lemma \ref{lem:incl}, so that we obtain the desired equality.
\end{proof}

\subsection{Implications for $\D_\psi$}
\label{sect:imp}

Fix a smooth algebraic group $\mbb{G}$ which admits a quasi-logarithm and an arbitrary positive integer $r$.  Let $G$ be the $r$-th Frobenius kernel in $\mbb{G}$.  We consider the Drinfeld double $\D=D(G)$.
\par

For such $G$, we have the corresponding infinitesimal group scheme $\Sigma=\Sigma(G,r)$ of Definition \ref{def:Sigma}, and for any $1$-parameter subgroup $\psi:\mbb{G}_{a(s),K}\to G_K$ we have an associated unipotent subgroup $\Sigma_\psi\subset \Sigma_K$.  By  Proposition \ref{prop:809}, any choice of quasi-logarithm for $\mbb{G}$ determines a compatible collection of linear equivalences
\begin{equation}\label{eq:896}
\mcl{L}:\rep(\D)\overset{\sim}\to \rep(\Sigma)\ \ \text{and}\ \ \mcl{L}_\psi:\rep(\D_\psi)\overset{\sim}\to \rep(\Sigma_\psi),
\end{equation}
which preserve the unit objects in the respective categories
\par

Since cohomological support for a local Hopf algebra depends only on the abelian structure on the representation category, we see that the diagram of \eqref{eq:soe_sch} restricts to a diagram
\begin{equation}\label{eq:soe_sch2}
\xymatrix{
|\D_\psi|_{V}\ar[rr]^{f_\psi} & & |\D|\\
|\Sigma_\psi|_{\mcl{L}_\psi V}\ar[u]^\cong_{\mcl{L}^\ast_\psi}\ar[rr]_{f'_\psi} & & |\Sigma|\ar[u]^\cong_{\mcl{L}^\ast},
}
\end{equation}
for any $V$ in $D^b(\D_\psi)$.  Hence the discussions of Subsections \ref{sect:loc1} and \ref{sect:lil_classify} imply the following.

\begin{proposition}\label{prop:tpp_Dp}
Let $G$ be as above, and fix an embedded $1$-parameter subgroup $\psi:\mbb{G}_{a(s),K}\to G_K$.  Then the following hold:
\begin{enumerate}
\item Thick ideals in $D^b(\D_\psi)$ are classified by cohomological support.
\item For any finite-dimensional $\D_\psi$-representations $V$ and $W$ we have
\[
|\D_\psi|_{(V\ot W)}=|\D_\psi|_V\cap |\D_\psi|_W.
\]
\end{enumerate}
\end{proposition}

\begin{proof}
From the linear equivalence $\mcl{L}_\psi$, Theorem \ref{thm:G_class}, and Lemma \ref{lem:614}, we understand that thick ideals in $D^b(\D_\psi)$ are classified by cohomological support, establishing (1).  A direct application of Proposition \ref{prop:tpp_A} now implies (2).
\end{proof}

\subsection{Restrictions of support and the tensor product property}
\label{sect:V}

As above, let $G$ be the $r$-th Frobenius kernel in a smooth algebraic group $\mbb{G}$, and suppose that $\mbb{G}$ admits a quasi-logarithm.

\begin{lemma}\label{lem:751}
Let $\mcl{L}:\rep(\D)\to \rep(\Sigma)$ be the linear equivalence induced by a choice of quasi-logarithm for $\mbb{G}$.  Then for any finite-dimensional $\D$-representation $V$ the isomorphism $\mcl{L}^\ast:|\Sigma|\overset{\cong}\to |\D|$ restricts to an isomorphism of supports $|\Sigma|_{\mcl{L}V}\overset{\cong}\to |\D|_{V}$.
\end{lemma}

\begin{proof}
Via the diagram of equivalences of Proposition \ref{prop:809}, and Theorem \ref{thm:proj_check}, we understand that a $\Sigma$-representation is projective if and only if its restriction to each of the $\Sigma_\psi$ is projective.  We can therefore repeat the proof of Proposition \ref{prop:reconstruct} to obtain a reconstruction of support
\[
|\Sigma|_W=\bigcup_{1\text{-param subgroups}}f'_\psi\left(|\Sigma_\psi|_{W_K}\right)
\]
for any $\Sigma$-representation $W$, where the $f'_\psi$ are the maps on projective spectra induced by restriction.
\par

The above expression, and the analogous expression of Proposition \ref{prop:reconstruct} therefore imply the claimed equality.  To argue this point more clearly, take a point $x\in |\Sigma|_{\mcl{L}V}$.  Then $x$ is in the image of some map $f'_\psi:|\Sigma_\psi|_{\mcl{L}_\psi V_K}\to |\Sigma|$.  It follows by the the diagram \eqref{eq:soe_sch2} that $\mcl{L}^\ast(x)\in |\D|$ is in the image of the corresponding map $f_\psi:|\D_\psi|_{V_K}\to |\D|$.  Hence $\mcl{L}^\ast(x)\in |\D|_V$.  This gives an inclusion $\mcl{L}^\ast(|\Sigma|_{\mcl{L}V})\subset |\D|_V$.  Since this argument is completely symmetric, we obtain the opposite inclusion as well and find that we have an identification $\mcl{L}^\ast(|\Sigma|_{\mcl{L}V})=|\D|_V$.
\end{proof}

Recall from Proposition \ref{prop:incl} that, for any embedded $1$-parameter subgroup $\psi$ which is defined over $k$, the map $f_\psi:|\D_\psi|\to |\D|$ is universally injective.  Furthermore, in this case $f_\psi$ is simply the map induced by restriction (i.e.\ it involves no base change).

\begin{proposition}\label{prop:restrict}
Consider any embedded $1$-parameter subgroup $\psi:\mbb{G}_{a(s)}\to G$ which is defined over $k$, and identify $|\D_\psi|$ with a closed subscheme in $|\D|$ via the map induced by restriction (Proposition \ref{prop:incl}).  Then for any finite-dimensional $\D$-representation $V$ we have
\[
|\D_\psi|_V=|\D_\psi|\cap |\D|_V.
\]
\end{proposition}

\begin{proof}
By the diagram \eqref{eq:soe_sch2}, and Lemma \ref{lem:751}, it suffices to check that we have an equality
\[
|\Sigma_\psi|_{W}=|\Sigma_\psi|\cap |\Sigma|_{W}
\]
for any finite-dimensional $\Sigma$-representation $W$.  However, the above equality follows from the analysis of support for infinitesimal group schemes given in \cite{suslinfriedlanderbendel97}--in particular \cite[Corollary 5.4.1, Proposition 7.4]{suslinfriedlanderbendel97}.
\end{proof}

We can now prove that cohomological support for the Drinfeld double $\D$ satisfies the tensor product property.

\begin{theorem}\label{thm:tensor}
Consider a Frobenius kernel $G=\mbb{G}_{(r)}$ in a smooth algebraic group $\mbb{G}$.  Suppose also that $\mbb{G}$ admits a quasi-logarithm.  Then for any finite-dimensional $\D$-representations $V$ and $W$ we have
\[
|\D|_{(V\ot W)}=|\D|_V\cap |\D|_W
\]
\end{theorem}

\begin{proof}
Consider any point in the intersection $x\in |\D|_V\cap |\D|_W$, and let $\psi:\mbb{G}_{a(s),K}\to G_K$ be any embedded $1$-parameter subgroup for which $x$ is in the image of the map $|\D_\psi|\to |\D|$.  Let $x'\in |\D_\psi|$ be any lift of $x$.  Since the support of $V_K$ (resp.\ $W_K$) over $\D_K$ is simply the preimage of $|\D|_V$ (resp. $|\D|_W$) along the projection $|\D_K|\to |\D|$, by Lemma \ref{lem:297}, we have $x'\in |\D_K|_{V_K}\cap |\D_K|_{W_K}$.  So, by changing base, we may assume that $x$ is the image of $|\D_\psi|$, for $\psi:\mbb{G}_{a(s)}\to G$ a $1$-parameter subgroup which is defined over $k$.
\par

Since $x$ is in $|\D|_V$, $|\D|_W$, and $|\D_\psi|$, Proposition \ref{prop:restrict} implies
\[
x\in |\D_\psi|_V\cap |\D_\psi|_W.
\]
By the tensor product property for $\D_\psi$, Proposition \ref{prop:tpp_Dp}, we then have $x\in |\D_\psi|_{(V\ot W)}$.  From the inclusion $|\D_\psi|_X\subset |\D|_X$, for arbitrary $X$, we see that $x$ is in $|\D|_{(V\ot W)}$ as well.  We therefore have an inclusion $(|\D|_V\cap |\D|_W)\subset |\D|_{(V\ot W)}$.
\par

For the opposite inclusion $|\D|_{(V\ot W)}\subset (|\D|_V\cap |\D|_W)$, one can restrict to some choice of $\D_\psi$ and argue similarly.  However, since the representation category $\rep(\D)$ is braided, this opposite inclusion actually comes for free.  See for example \cite[Proposition 3.3]{berghplavnikwitherspoon21}.
\end{proof}

\section{Support and tensor products for infinite-dimensional representations}
\label{sect:bigtensor}

We consider support for infinite-dimensional representations over the Drinfeld double $\D=D(G)$.  The support theory which we employ is a kind of ``hybrid theory", which we produce via the restriction functors $\rep(\D)\to \rep(\D_\psi)$ and the Benson-Iyengar-Krause (local cohomology) support theory for the $\D_\psi$.  We prove that this hybrid support theory detects projectivity of arbitrary $\D$-representations, and admits a sufficiently strong tensor product property.
\par

The results of this section provide the necessary foundations for our analysis of thick ideals in the (small) stable category $\stab(\D)$ in Section \ref{sect:thicK}.

\subsection{Stable categories}
Let $A$ be a finite-dimensional Hopf algebra.  We consider the stable categories $\stab(A)$ and $\Stab(A)$ for $A$.  These are the quotient categories of $\rep(A)$ and $\Rep(A)$, respectively, by the tensor ideal consisting of all morphisms which factor through a projective.
\par

In addition to the derived category $D^b(A)$ of finite-dimensional representations over $A$, we consider
\[
D^b_{\rm big}(A)=\{\text{The bounded derived category of arbitrary $A$-representations}\}.
\]
We have canonical equivalences to the Verdier quotients
\[
\stab(A)\overset{\sim}\to D^b(A)/\langle \operatorname{proj}(A)\rangle,\ \ \Stab(A)\overset{\sim}\to D^b_{\rm big}(A)/\langle \Proj(A)\rangle
\]
\cite{rickard89}, which provide the stable categories with triangulated structures.  These equivalences also provide actions of the extension algebra $\Ext_A^\ast(k,k)$ on the stable representation categories
\[
-\ot M:\Ext^\ast_A(k,k)\to \Hom^\ast_{\Stab}(M,M)\ \ \forall\ M\in \Stab(A).
\]
\par

The inclusion $\stab(A)\to \Stab(A)$ is exact and fully faithful, and identifies the small stable category with the subcategory of compact objects in $\Stab(A)$.

\subsection{Local cohomology support}

Let $A$ be a finite-dimensional Hopf algebra with finite type cohomology.  We suppose additionally that cohomological support for finite-dimensional $A$-representations satisfies the inclusion
\begin{equation}\label{eq:848}
|A|_{V\ot W}\subset \big(|A|_V\cap |A|_W\big).
\end{equation}
For example, we might consider $A$ to be a local Hopf algebra with finite type cohomology (see Lemma \ref{lem:incl}).
\par

Take $E_A:=\Ext^\ast_A(k,k)$.  As remarked above, we have natural actions of $E_A$ on objects in the big stable category $\Stab(A)$, which collectively constitute a map to the graded center $E_A\to Z(\Stab(A))=\operatorname{End}_{\operatorname{Fun}}(id_{\Stab(A)})$.  Given this situation, we can consider the local cohomology support of Benson, Iyengar, and Krause \cite{bensoniyengarkrause08}.  This support theory is defined via certain triangulated endofunctors $\Gamma_p:\Stab(A)\to \Stab(A)$ associated to (arbitrary) points in the projective spectrum $|A|=\Proj(E_A)$.  We have specifically
\begin{equation}\label{eq:lc_supp}
\supp^{lc}_A(M):=\{p\in |A|: \Gamma_p(M)\neq 0\}
\end{equation}
\cite[\S 5.1]{bensoniyengarkrause08}.  We note that the points $p$ appearing in the above formula are not necessarily closed, and that supports of objects in $\Stab(A)$ are not necessarily closed in $|A|$.
\par

Since the support theory \eqref{eq:lc_supp} is defined via the vanishing of certain triangulated endofunctors, it behaves appropriately under sums, shifts, and exact triangles.  Specifically, the support of a sum $M\oplus M'$ is the union of the supports of $M$ and $M'$, support is invariant under the shift automorphism, and the support of an object $N$ which fits into a triangule $M\to N\to M'\to \Sigma M$ is contained in the union $\supp^{lc}_A(M)\cup \supp_A^{lc}(M')$.
\par

By pulling back along the quotient $D^b_{\rm big}(A)\to \Stab(A)$, we may consider local cohomology support $\supp^{lc}_A$ as a support theory which also takes $A$-complexes as inputs as well.

\begin{theorem}[\cite{bensoniyengarkrause08}]\label{thm:properties}
For $A$ as above, the following hold:
\begin{enumerate}
\item $M$ vanishes in $\Stab(A)$ if and only if $\supp^{lc}_A(M)=\emptyset$.
\item For any object $V$ in $D^b(A)$, i.e.\ any bounded complex of finite-dimensional representations, there is an equality of supports
\[
\supp^{lc}_A(V)=|A|_V.
\]
\item For arbitrary $M$ and $N$ in $D^b_{\rm big}(A)$, local cohomology support satisfies
\[
\supp^{lc}_A(M\ot N)\subset \big(\supp^{lc}_A(M)\cap\supp^{lc}_A(N)\big).
\]
\end{enumerate}
\end{theorem}

\begin{proof}
Statements (1) and (2) are covered in \cite[Theorem 5.13]{bensoniyengarkrause08}.  For the claimed inclusion (3), we note that for any specialization closed subset $\Theta\subset |A|$ the containment \eqref{eq:848} tells us that the subcategory
\[
\msc{K}_\Theta:=\{V\text{ in }\stab(A):|A|_V\subset \Theta\}
\]
is a thick ideal in $\stab(A)$.  Thus one follows the proof of \cite[Theorem 8.2]{bensoniyengarkrause08} to see that
\[
\Gamma_p(M\ot N)=M\ot \Gamma_p(N)=\Gamma_p(M)\ot N.
\]
From the above equation, and the definition of the support $\supp^{lc}_A$, we deduce the inclusion of (3).
\end{proof}

\subsection{$\bpsi$-local support for $\D$-representations}
\label{sect:psi_loc}

Consider an infinitesimal group scheme $G$, with associated Drinfeld double $\D=D(G)$.  Let $M$ be an object in the bounded derived category $D^b_{\rm big}(\D)$ of arbitrary $\D$-representations, and recall the maps $f_\psi:|\D_\psi|\to |\D|$ induced by restriction \eqref{eq:f_psi}.  We define the support
\begin{equation}\label{eq:psi-loc}
\supp^{\ploc}(M):=\bigcup_{1\text{-param subgroups}}f_\psi\left(\supp^{lc}_{\D_\psi}(\res_\psi M_K)\right),
\end{equation}
where the union runs over all embedded $1$-parameter subgroups $\psi:\mbb{G}_{a(s),K}\to G_K$, and $\res_\psi:\rep(G_K)\to \rep(\mbb{G}_{a(s),K})$ denotes the restriction functor.  As in Proposition \ref{prop:reconstruct}, \eqref{eq:psi-loc} defines the support $\supp^{\ploc}(M)$ as a union of sub\emph{sets} in the projective spectrum of cohomology $|\D|$.
\par

We refer to the support \eqref{eq:psi-loc} as the $\bpsi$-local support of $M$.  Note that this support takes values in the projective spectrum of cohomology $|\D|$.  By pulling back along the quotient map
\[
D^b_{\rm big}(\D)\to \Stab(\D)
\]
we freely consider the $\bpsi$-local support as a support theory for the bounded derived category of arbitrary $\D$-representations as well.

\begin{remark}
We have used a boldface $\bpsi$ in our notation to indicate that $\bpsi$ might be thought of as a coordinate which ranges over the space of $1$-parameter subgroups.
\end{remark}

We list some basic properties of $\bpsi$-local support.

\begin{lemma}\label{lem:psi-loc_props}
For any infinitesimal group scheme $G$, $\bpsi$-local support satisfies the following:
\begin{itemize}
\item $\supp^{\ploc}(M)=\emptyset$ if and only if $M$ vanishes in the stable category $\Stab(\D)$.\vspace{1mm}
\item $\supp^{\ploc}(M\oplus N)=\supp^{\ploc}(M)\cup\supp^{\ploc}(N)$.\vspace{1mm}
\item For any triangle $M\to N\to M'$,
\[
\supp^{\ploc}(N)\subset \left(\supp^{\ploc}(M)\cup \supp^{\ploc}(M')\right).
\]
\item $\supp^{\ploc}(M\ot N)\subset \left(\supp^{\ploc}(M)\cap \supp^{\ploc}(N)\right)$.\vspace{1mm}
\item $\supp^{\ploc}(\Sigma M)=\supp^{\ploc}(M)$.\vspace{1mm}
\item For any $V$ in $D^b(\D)$, $\supp^{\ploc}(V)=|\D|_V$.
\end{itemize}
\end{lemma}

In the above formulas $M$, $M'$, and $N$ are arbitrary objects in $D^b_{\rm big}(\D)$.

\begin{proof}
The first point follows by the projectivity test of Theorem \ref{thm:proj_check}, and the detection propert for local cohomology support over $\D_\psi$.  The four subsequent points follow directly from the corresponding properties for the local cohomology supports $\supp^{lc}_{\D_\psi}$, and the fact that restriction is an exact tensor functor.  The final point follows from the identification $\supp^{lc}_{\D_\psi}(V_K)=|\D_\psi|_{V_K}$ and the reconstruction formula of Proposition \ref{prop:reconstruct}.
\end{proof}

\subsection{$\bpsi$-local support and tensor products}

\begin{theorem}\label{thm:bigtpp}
Consider a Frobenius kernel $G$ in a smooth algebraic group $\mbb{G}$.  Suppose that $\mbb{G}$ admits a quasi-logarithm.  Then for any object $V$ in $D^b(\D)$, and any $M$ in $D^b_{\rm big}(\D)$, we have
\begin{equation}\label{eq:970}
\supp^{\ploc}(V\ot M)=\supp^{\ploc}(V)\cap \supp^{\ploc}(M).
\end{equation}
\end{theorem}

Note that, since $\Rep(\D)$ is a braided monoidal category, an identification \eqref{eq:970} implies the corresponding equality for the action of finite-dimensional representations (or complexes) on the right
\[
\supp^{\ploc}(M\ot V)=\supp^{\ploc}(M)\cap \supp^{\ploc}(V),
\]
simply because $V\ot M\cong M\ot V$.  In the language of \cite[Definition 4.7]{negronpevtsova}, we are claiming that cohomological support for $\D$ is a \emph{lavish support theory} for the stable category $\stab(\D)$.
\par

Before proving Theorem \ref{thm:bigtpp}, we prove its local analog.

\begin{proposition}\label{prop:bigtpp}
Let $G$ be as in the statment of Theorem \ref{thm:bigtpp}, and consider an embedded $1$-parameter subgroup $\psi:\mbb{G}_{a(s)}\to G$ which is defined over $k$.  Then for $W$ in $D^b(\D_\psi)$, and $N$ in $D^b_{\rm big}(\D_\psi)$, local cohomology support satisfies
\[
\supp^{lc}_{\D_\psi}(W\ot N)=\supp^{lc}_{\D_\psi}(W)\cap \supp^{lc}_{\D_\psi}(N).
\]
\end{proposition}

\begin{proof}
It suffices to prove the inclusion
\[
\supp^{lc}_{\D_\psi}(W)\cap \supp^{lc}_{\D_\psi}(N)\subset \supp^{lc}_{\D_\psi}(W\ot N),
\]
since the opposite inclusion follows by Theorem \ref{thm:properties}.  Since the local cohomology support is defined via the vanishing of the exact endomorphisms $\Gamma_p$, we understand that if $Q'$ in $\Stab(\D_\psi)$ is in the thick subcategory generated by $Q$ then $\supp^{lc}_{\D_\psi}(Q')\subset \supp^{lc}_{\D_\psi}(Q)$.  So it suffices to prove that there is an equality
\[
\supp^{lc}_{\D_\psi}(W)\cap \supp^{lc}_{\D_\psi}(N)=\supp^{lc}_{\D_\psi}(L\ot N)
\]
for some $L$ in the thick subcategory generated by $W$ in $\stab(\D_\psi)$.
\par

Let $L$ be a product of Carlson modules such that $\supp^{lc}_{\D_\psi}(L)=\supp^{lc}_{\D_\psi}(W)$.  By Lemma \ref{lem:614} and Proposition \ref{prop:tpp_Dp}, the object $L$ is in the thick subcategory generated by $V$ in $\stab(\D_\psi)$ and thus $L\ot N$ is in the thick subcategory generated by $W\ot N$ in $\Stab(\D_\psi)$.
\par

Recall that, in the stable category, the Carlson module $L_\zeta$ associated to an extension $\zeta:k\to \Sigma^nk$ is isomorphic to a shift of the mapping cone $\operatorname{cone}(\zeta)$.  So by \cite[Lemma 2.6]{bensoniyengarkrause11II} we have
\[
\supp^{lc}_{\D_\psi}(L_\zeta\ot N)=Z(\zeta)\cap \supp^{lc}_{\D_\psi}(N)=\supp^{lc}_{\D_\psi}(L_\zeta)\cap \supp^{lc}_{\D_\psi}(N)
\]
for any such $L_\zeta$.  It follows that, for our product of Carlson modules $L$, we have
\[
\supp^{lc}_{\D_\psi}(L\ot N)=\supp^{lc}_{\D_\psi}(L)\cap \supp^{lc}_{\D_\psi}(N)=\supp^{lc}_{\D_\psi}(W)\cap \supp^{lc}_{\D_\psi}(N),
\]
as desired.
\end{proof}

We now prove our theorem.

\begin{proof}[Proof of Theorem \ref{thm:bigtpp}]
We have already observed one inclusion in Lemma \ref{lem:psi-loc_props}.  So we need only establish the inclusion
\begin{equation}\label{eq:977}
\supp^{\ploc}(V)\cap \supp^{\ploc}(M)\subset \supp^{\ploc}(V\ot M).
\end{equation}
\par

Consider any point $x$ in the above intersection, and choose an embedded subgroup $\psi:\mbb{G}_{a(s),K}\to G_K$ for which $x$ is the image of a point $x'\in \supp^{lc}_{\D_\psi}(M_K)$.  The naturality property
\[
\supp^{lc}_{\D_\psi}(V_K)=|\D_\psi|\cap |\D_K|_{V_K}
\]
of Proposition \ref{prop:restrict} implies that $x'$ is in $\supp^{lc}_{\D_\psi}(V_K)$ as well.  (See also Lemma \ref{lem:297}.)  We apply the equality
\[
\supp^{lc}_{\D_\psi}(V_K\ot M_K)=\supp^{lc}_{\D_\psi}(V_K)\cap \supp^{lc}_{\D_\psi}(M_K)
\]
of Proposition \ref{prop:bigtpp} to see that $x'\in \supp^{lc}_{\D_\psi}(V_K\ot M_K)$, and hence $x\in \supp^{\ploc}(V\ot M)$ by the definition of the $\bpsi$-local support.  We thus verify the inclusion \eqref{eq:977}, and obtain the proposed tensor product property.
\end{proof}

\section{Thick ideals and the Balmer spectrum}
\label{sect:thicK}

We provide a classification of thick ideals in the stable category $\stab(\D)$, for $\D$ the Drinfeld double of an appropriate Frobenius kernel.  We then apply results of Balmer to calculate the spectrum of prime ideals in the stable category $\stab(\D)$.  In particular, we show that thick ideals are classified by specialization closed subsets in the projective spectrum of cohomology $|\D|$, and we show that the Balmer spectrum is isomorphic to cohomological spectrum $|\D|$ as a locally ringed space.

\subsection{Classification of thick ideals and prime ideal spectra}
\label{sect:ideals}

Let $\D$ be the Drinfeld double of a finite group scheme.  Recall that a specialization closed subset $\Theta$ in $|\D|=\Proj\Ext^\ast_{\D}(k,k)$ is a subset which contains the closures of all of its points.  Equivalently, a specialization closed subset is an arbitrary union of closed subsets in $|\D|$.
\par

For any specialization closed subset $\Theta$ in $|\D|$ we have the associated thick ideal
\[
\msc{K}_\Theta:=\{V\in \stab(\D):|\D|_V\subset \Theta\}
\]
in the stable category $\stab(\D)$.  To see that $\msc{K}_\Theta$ is in fact closed under the tensor actions $\stab(\D)$ on the left and right, one simply consults the inclusion $|\D|_{V\ot W}\subset (|\D|_V\cap |\D|_W)$ provided by the braiding on $\rep(\D)$ \cite[Proposition 3.3]{berghplavnikwitherspoon21}.  Similarly, for any thick ideal $\msc{K}\subset \stab(\D)$ we have the associated support space
\[
|\D|_{\msc{K}}:=\cup_{V\in \msc{K}}|\D|_{V},
\]
which is a specialization closed subset in $|\D|$.  We note that the formal properties of cohomological support imply an equality $|\D|_V=|\D|_{\langle V\rangle_\ot}$ between the support of a given object $V$, and the support of the thick ideal $\langle V\rangle_\ot$ which it generates in $\stab(\D)$.
\par

The two above operations define maps of sets
\begin{equation}\label{eq:class}
\{\text{thick ideals in }\stab(\D)\}\underset{\msc{K}_?}{\overset{|\D|_?}\leftrightarrows} \{\text{specialization closed subsets in }|\D|\}
\end{equation}
which preserve the respective orderings by inclusion.  In rephrasing Definition \ref{def:lil_classify}, we say cohomological support for $\D$ \emph{classifies thick ideals} in $\stab(\D)$ if the two maps in \eqref{eq:class} are mutually inverse bijections.
\par

At this point it is a formality to deduce a classification of thick ideals in the stable category $\stab(\D)$ from the support theoretic results of Lemma \ref{lem:psi-loc_props} and Theorem \ref{thm:bigtpp}.  One can see for example \cite{rickard97}.  We follow the generic presentation of \cite{negronpevtsova}.

\begin{theorem}\label{thm:thick_id}
Consider a smooth algebraic group $\mbb{G}$ which admits a quasi-logarithm, and let $G$ be a Frobenius kernel in $\mbb{G}$.  Then, for the Drinfeld double $\D=D(G)$, cohomological support classifies thick ideals in the stable category $\stab(\D)$.  That is to say, the two maps of \eqref{eq:class} are mutually inverse bijections.
\end{theorem}

\begin{proof}
Theorem \ref{thm:bigtpp} tells us that cohomological support is a lavish support theory for $\stab(\D)$, in the language of \cite[\S 4.3]{negronpevtsova}.  So the claimed classification follows by \cite[Proposition 5.2]{negronpevtsova}.
\end{proof}

We note that, by pulling back along the projection $\pi:D^b(\D)\to \stab(\D)$, we can similarly use cohomology to classify thick ideals in the bounded derived category for $\D$.  Namely, under the map $\pi$ thick ideals in $\stab(\D)$ are identified with thick ideals in $D^b(\D)$ which contain the ideal $\operatorname{perf}(\D)$ of bounded complexes of projectives.  This subcollection of ideals in $D^b(\D)$ is precisely the collection of nonvanishing ideals in $D^b(\D)$.  So we obtain a classification
\[
\{\text{thick ideals in }D^b(\D)\}\cong \{\text{specialization closed subsets in }|\D|\}\cup\{0\}.
\]

\subsection{Prime ideal spectra for Drinfeld doubles}
\label{sect:spectrum}

Consider again the Drinfeld double $\D$ of a finite group scheme $G$.

We recall that the sublattice of thick \emph{prime} ideals in $\stab(\D)$ forms a locally ringed space, which is referred to as the Balmer spectrum
\begin{equation}\label{eq:1078}
\Spec(\stab(\D)):=\left\{\begin{array}{c}
\text{the collection of thick prime ideals in $\stab(\D)$}\\
\text{with the topology and ringed structure described in \cite{balmer05}}
\end{array}\right\}.
\end{equation}
As one might expect, by a thick prime ideal in $\stab(\D)$ we mean a proper thick ideal $\msc{P}$ for which an inclusion $V\ot W\in \msc{P}$ implies either $V\in \msc{P}$ or $W\in \msc{P}$.  We do not recall the topology or the ringed structure on the spectrum here, and refer the reader instead to the highly readable text \cite[\S 1, \S 6]{balmer05}.
\par

As explained in \cite{balmer05,balmer10}, a classification of thick ideals in $\stab(\D)$ via cohomological support implies a corresponding calculation of the prime ideal spectrum.

\begin{theorem}\label{thm:spec}
For $G$ as in Theorem \ref{thm:thick_id}, there is a homeomorphism
\[
f_{\rm coh}:|\D|=\Proj\Ext^\ast_\D(k,k)\overset{\cong}\longrightarrow \Spec(\stab(\D))
\]
defined by taking $f_{coh}(x)=\{V\in \stab(\D):x\notin |\D|_V\}$.  Furthermore, $f_{\rm coh}$ can be upgraded to an isomorphism of locally ringed spaces.
\end{theorem}

\begin{proof}
Given Theorem \ref{thm:thick_id}, the fact that $f_{\rm coh}$ is a homeomorphism follows from \cite[Theorem 5.2]{balmer05}.  By \cite[Proposition 6.11]{balmer10}, the homeomorphism $f_{\rm coh}$ furthermore enhances to an isomorphism of locally ringed spaces.  To elaborate, in \cite[Definition 5.1, 6.10]{balmer10} a map of ringed spaces $\rho:\Spec(\stab(\D))\to |\D|$ is constructed.  One sees directly that the composite $\rho\circ f_{\rm coh}:|\D|\to |\D|$ is the identity, as a map of topological spaces.  Since $f_{\rm coh}$ is a homeomorphism, we see that $\rho$ is a homeomorphism as well.  It follows by \cite[Proposition 6.11]{balmer10} that $\rho$ is an isomorphism of (locally) ringed spaces, and so provides the homeomorphism $f_{\rm coh}=\rho^{-1}$ with ringed structure under which it is also an isomorphism of locally ringed spaces.
\end{proof}

\begin{remark}
In \cite{balmer05,balmer10} Balmer only considers symmetric tensor triangulated categories.  However, all of the definitions, results, and proofs from \cite{balmer05,balmer10} apply verbatim in the braided context.  So, implicitly, we use the fact that $\rep(\D)=Z(\rep(G))$ admits a canonical (highly non-symmetric!) braided structure in the definition \eqref{eq:1078}, and also in the proof of Theorem \ref{thm:spec}.  One can alternatively refer to \cite[\S 6]{negronpevtsova} and in particular \cite[Theorem 6.10]{negronpevtsova}.
\end{remark}


\appendix

\section{A $\pi$-point rank variety for the Drinfeld double}
\label{sect:pi}

We introduce a $\pi$-point rank variety $\Pi(\D)$ for the Drinfeld double $\D$, whose points consist of certain classes of flat algebra maps $K[t]/(t^p)\to \D_K$.  For any $\D$-representation $V$ we construct an associated support space $\Pi(\D)_V$ in $\Pi(\D)$.  We show that the support theory $V\mapsto \Pi(\D)_V$ behaves in the expected manner when we consider the Drinfeld double of a Frobenius kernel $G=\mbb{G}_{(r)}$ in a sufficiently nice algebraic group $\mbb{G}$.  In particular, the support space $\Pi(\D)_V$ vanishes if and only if the given representation $V$ is projective, and the support spaces satisfy the tensor product property
\[
\Pi(\D)_{V\ot W}=\Pi(\D)_V\cap \Pi(\D)_W.
\]
Furthermore, we establish an identification with cohomological support $\Pi(G)_\star\overset{\cong}\to |\D|_\star$.  We also show that our $\pi$-support can be identified with a certain ``universal" $\pi$-point support, which we define in Section \ref{sect:Pi_ot}.
\par

Since these results of this section are isolated from those of the body of the text, in a technical sense, we collect them here in an appendix.

\subsection{$\pi$-points and support for finite group schemes}

Throughout this subsection $G$ is a finite group scheme over our base field $k$.  We recall some definitions and results from \cite{friedlanderpevtsova07}.

\begin{definition}
A $\pi$-point for a finite group scheme $G$, over $k$, is a pair of a field extension $k\to K$ and a flat algebra map $\alpha:K[t]/(t^p)\to KG$ which factors through the group ring of an abelian, unipotent subgroup $U\subset G_K$.
\end{definition}

We generally abuse notation and simply write $\alpha$ for the pair $(K/k, \alpha)$.  Any $\pi$-point defines a corresponding point $p_\alpha$ in the projective spectrum of cohomology $|G|$, which is explicitly the homogenous prime ideal
\begin{equation}\label{eq:870}
p_\alpha:=\ker\left(\Ext^\ast_G(k,k)\overset{K\ot -}\to \Ext^\ast_{G_K}(K,K)\overset{\res_\alpha}\longrightarrow \Ext^\ast_{K[t]/t^p}(K,K)_{\rm red}=K[T]\right).
\end{equation}
In the above formula $T$ is a variable of cohomological degree 2 (or $1$ in characteristic 2).  Flatness of the extension $\alpha$ ensures that the ideal $p_\alpha$ is not all of $\Ext^{>0}_G(k,k)$, so that $p_\alpha$ does in fact define a point in the projective spectrum \cite[Lemma 3.4]{friedlanderpevtsova05} (cf.\ \cite[Theorem 3.2.1]{aapw}).

\begin{definition}
For a given finite group scheme $G$, we say two $\pi$-points $\alpha:K[t]/(t^p)\to KG$ and $\beta:L[t]/(t^p)\to LG$ are equivalent if any finite-dimensional $G$-representation $V$ which restricts to a projective $K[t]/(t^p)$-representation $\res_\alpha(V_K)$ along $\alpha$ also restricts to a projective $L[t]/(t^p)$-representation $\res_\beta(V_L)$ along $\beta$, and vice versa.
\par

We let $\Pi(G)$ denote the collection of equivalence classes of $\pi$-points
\[
\Pi(G)=\{[\alpha]:\alpha:K[t]/(t^p)\to KG\ \text{is a $\pi$-point for }G\}.
\]
For any finite-dimensional $G$-representation $V$ we define the $\pi$-support space $\Pi(G)_V$ as
\[
\Pi(G)_V=\{[\alpha]:\res_\alpha(V_K)\ \text{is non-projective over }K[t]/(t^p)\}.
\]
\end{definition}

The collection of subsets $\{\Pi(G)_V: V\in \rep(G)\}$ in $\Pi(G)$ is closed under finite unions, since $\Pi(G)_V\cup\Pi(G)_W=\Pi(G)_{V\oplus W}$.  Hence there is a uniquely defined topology on $\Pi(G)$ for which the supports of objects $\Pi(G)_V$ provide a basis of closed subsets.

\begin{theorem}[{\cite[Theorem 3.6]{friedlanderpevtsova07}}]\label{thm:891}
If two $\pi$-points $\alpha$ and $\beta$ for $G$ are equivalent, then the corresponding points $p_\alpha,p_\beta\in |G|$ are equal.  Furthermore, the resulting map
\[
\Pi(G)\to |G|,\ \ [\alpha]\mapsto p_\alpha
\]
is a homeomorphism, and for any finite-dimensional representation $V$ this homeomorphism restricts to a homeomorphism $\Pi(G)_V\to |G|_V$.
\end{theorem}

Note that Theorem \ref{thm:891} tells us that the topological space $\Pi(G)$ is Noetherian.  Hence the basic closed sets $\{\Pi(G)_V\}_{V\in \rep(G)}$ in $\Pi(G)$ provide the collection of \emph{all} closed sets in $\Pi(G)$ \cite[Proposition 3.4]{friedlanderpevtsova07}.

\begin{remark}
One of the main advancements of \cite{friedlanderpevtsova07} is the observation that one can reasonably define support spaces $\Pi(G)_M$ for \emph{infinite-dimensional} $G$-representation $M$.  So, the above presentation omits some of the more significant aspects of \cite{friedlanderpevtsova07}.  One can see Remark \ref{rem:strong} below for additional context.
\end{remark}

\subsection{$\pi$-point support for $\D_\psi$}

We consider an infinitesimal group scheme $G$, with corresponding Drinfeld double $\D=D(G)$.

\begin{definition}\label{def:pi_psi}
Consider any infinitesimal group scheme $G$, and fix an embedded $1$-parameter subgroup $\psi:\mbb{G}_{a(s)}\to G$ which is defined over $k$.  A $\pi$-point for $\D_\psi$ is a pair of a field extension $k\to K$, and a flat algebra map $\alpha:K[t]/(t^p)\to (\D_\psi)_{K}$ such that
\begin{enumerate}
\item[(a)] there exists an \emph{algebra} identification $\D_\psi=kH$ between $\D_\psi$ and the group algebra of a finite group scheme $H$ over $k$.
\item[(b)] under some identification as in (a), $\alpha$ corresponds to a $\pi$-point for the given group scheme $H$.
\end{enumerate}
\end{definition}

Statements (a) and (b) above can alternately be stated as follows: a $\pi$-point for $\D_\psi$ is a flat algebra map $\alpha:K[t]/(t^p)\to (\D_\psi)_K$ which is a $\pi$-point for $\D_\psi$ relative to some \emph{alternate} choice of \emph{cocommutative} Hopf structure $\Delta'$ on $\D_\psi$.  We note that any group scheme $H$ as in (a) is necessarily unipotent, since $\D_\psi$ is local.
\par

We say two $\pi$-points $\alpha:K[t]/(t^p)\to (\D_\psi)_K$ and $\beta:L[t]/(t^p)\to (\D_\psi)_L$ for $\D_\psi$ are equivalent if any finite-dimensional $\D_\psi$-representation $V$ with projective restriction $\res_\alpha(V_K)$ also has projective restriction $\res_\beta(V_L)$, and vice versa.  We define the $\pi$-point space in the expected manner
\[
\Pi(\D_\psi)=\{[\alpha]:\alpha:K[t]/(t^p)\to (\D_\psi)_K\text{ is a $\pi$-point}\},
\]
and for any finite-dimensional $\D_\psi$-representation $V$ we define the $\pi$-support space
\[
\Pi(\D_\psi)_V=\{[\alpha]:\res_\alpha(V_K)\text{ is non-projective over }K[t]/(t^p)\}.
\]
We note that if $\D_\psi$ admits no such identification with a group algebra $kH$, as required in Definition \ref{def:pi_psi} (a), then the space $\Pi(\D_\psi)$ is necessarily empty.

We topologize the space $\Pi(\D_\psi)$ via the basis of closed sets $\{\Pi(\D_\psi)_V: V\text{ in }\rep(\D_\psi)\}$.  As in \eqref{eq:870}, one sees that each $\pi$-point $\alpha$ defines a corresponding point $p_\alpha$ in the cohomological support space $|\D_\psi|$.

\begin{lemma}\label{lem:947}
If two $\pi$-points $\alpha$ and $\beta$ for $\D_\psi$ are equivalent, then their corresponding points $p_\alpha$ and $p_\beta$ in $|\D_\psi|$ are equal.  Furthermore, whenever the $\pi$-point space $\Pi(\D_\psi)$ is non-empty, the map
\[
\Pi(\D_\psi)\to |\D_\psi|,\ \ [\alpha]\mapsto p_\alpha
\]
is a homeomorphism and for any finite-dimensional $\D_\psi$-representation $V$ this homeomorphism restricts to a homeomorphism $\Pi(\D_\psi)_V\to |\D_\psi|_V$.
\end{lemma}

\begin{proof}
If $\D_\psi$ admits no cocommutative Hopf structure then the space $\Pi(\D_\psi)$ is empty, and there is nothing to prove.  So let us suppose that $\D_\psi$ admits the necessary alternate Hopf structure.
\par

Consider any cocommutative Hopf structure $\Delta'$ on the underlying algebra $\D_\psi$, and corresponding identification $\D_\psi=kH$.  Since $H$ is necessarily unipotent, as $\D_\psi$ is local, the cohomological support spaces agree $|H|_V=|\D_\psi|_V$ for all $V$ in $\rep(\D_\psi)=\rep(H)$.  (See Section \ref{sect:loc1}.)
\par

Now, Theorem \ref{thm:891} tells us that a $H$-representation $V$ is non-projective at a $\pi$-point $\alpha'$ for $H$ if and only if $p_{\alpha'}\in |H|_V$.  So by the above information we see that a $\D_\psi$-representation $V$ is non-projective at a $\pi$-point $\alpha$ if and only if $p_\alpha\in |\D_\psi|_V$.  Hence two $\pi$-points $\alpha$ and $\beta$ for $\D_\psi$ are equivalent if and only if $p_\alpha=p_\beta$.  This shows that the map $\Pi(\D_\psi)\to |\D_\psi|$ is well-defined and injective.  The map is furthermore surjective since, if we consider our identification $\D_\psi=kH$, the map $\Pi(H)\to |H|(=|\D_\psi|)$ is surjective, meaning every point in the cohomological support space is represented by a $\pi$-point $\alpha:K[t]/(t^p)\to KH=(\D_\psi)_K$.
\end{proof}

Based on the presentation of Section \ref{sect:soe}, we understand that $\D_\psi$ admits a cocommutative Hopf structure whenever $G$ is a Frobenius kernel in a smooth algebraic group which admits a quasi-logarithm.  So Lemma \ref{lem:947} tells us that we have an identification of support theories $\Pi(\D_\psi)_\star\cong |\D_\psi|_\star$ in this case.  In particular, the above lemma is not vacuous.

\subsection{$\pi$-point support for $\D$}

Fix an infinitesimal group scheme $G$ and $\D=D(G)$.

\begin{definition}\label{def:1445}
A $\pi$-point $\alpha$ for $\D$ is a pair of an embedded $1$-parameter subgroup $\psi:\mbb{G}_{a(s),K}\to G_K$ and a $\pi$-point $\bar{\alpha}:K[t]/(t^p)\to \D_\psi$, defined as in Definition \ref{def:pi_psi}.
\end{definition}

For any given $\pi$-point $(\psi,\bar{\alpha})$, we are particularly concerned with the composition $K[t]/(t^p)\to \D_K$ of the map $\bar{\alpha}:K[t]/(t^p)\to \D_\psi$ with the inclusion $\D_\psi\to \D$.  So we generally identify a $\pi$-point with its associated flat map $K[t]/(t^p)\to \D_K$, and simply write $\alpha:K[t]/(t^p)\to \D_K$ by an abuse of notation.

\begin{definition}\label{def:equiv}
Two $\pi$-points $\alpha:K[t]/(t^p)\to \D_K$ and $\beta:L[t]/(t^p)\to \D_L$ are said to be equivalent if any finite-dimensional representation $V$ which restricts to a projective $K[t]/(t^p)$-representation $\res_\alpha(V_K)$ along $\alpha$ also restricts to a projective $L[t]/(t^p)$-representation $\res_\beta(V_L)$ along $\beta$, and vice versa.
\end{definition}

We define the space of equivalence classes of $\pi$-points
\[
\Pi(\D)=\{[\alpha]:\alpha:K[t]/(t^p)\to \D_K\text{ is a $\pi$-point}\},
\]
and for any finite-dimensional $\D$-representation $V$ we define the $\pi$-support
\[
\Pi(\D)_V=\{[\alpha]:\res_\alpha(V_K)\text{ is non-projective}\}.
\]
The space $\Pi(\D)$ is topologized via the basis of closed sets provided by the supports $\Pi(\D)_V$ of all finite-dimensional $\D$-representations.
\par

As in \eqref{eq:870}, any $\pi$-point $\alpha:K[t]/(t^p)\to \D_K$ defines an associated point $p_\alpha\in |\D|$ in the cohomological support space.  One employs Carlson modules exactly as in \cite[Proposition 2.9]{friedlanderpevtsova07} to see that the two points $p_\alpha$ and $p_\beta$ agree whenever $\alpha$ and $\beta$ are equivalent.  So we find

\begin{proposition}\label{prop:973}
There is a well-defined continuous map
\[
w:\Pi(\D)\to |\D|,\ \  \alpha\mapsto p_\alpha.
\]
For any finite-dimensional $\D$-representation $V$, the above map restricts to a map between support spaces $\Pi(\D)_V\to |\D|_V$.
\end{proposition}

\begin{proof}
As stated above, well-definedness can be argued as in \cite{friedlanderpevtsova07}.  The fact that $\Pi(\D)_V$ is mapped to $|\D|_V$ can be reduced to the corresponding claim for $\pi$-support over the $\D_\psi$, which is covered in Lemma \ref{lem:947}.
\par

All that is left is to establish continuity of $w$.  For continuity, we note that any closed set in $|\D|$ is the support $|\D|_L$ of a product of Carlson modules.  The naturality properties of Lemma \ref{lem:159} then gives $w^{-1}(|\D|_L)=\Pi(\D)_L$.  This shows that the preimage of any closed set in $|\D|$ along $w$ is closed in $\Pi(\D)$.
\end{proof}

One can see from Theorem \ref{thm:proj_check}, and the arguments used in the proof of Proposition \ref{prop:973}, that the map $\Pi(\D)\to |\D|$ is in fact \emph{surjective} when $G$ is a Frobenius kernel in a sufficiently nice algebraic group $\mbb{G}$.  We leave the details to the interested reader, as we will observe a stronger result in Theorem \ref{thm:Pi} below.  As a related finding, we have the following.

\begin{theorem}\label{thm:pi_proj_check}
Suppose that $G$ is a Frobenius kernel in an algebraic group $\mbb{G}$, and that $\mbb{G}$ admits a quasi-logarithm.  Then a given finite-dimensional $\D$-representation $V$ is projective if and only if $\Pi(\D)_V=\emptyset$.
\end{theorem}

\begin{proof}
By Theorem \ref{thm:proj_check}, $V$ is projective if and only if its restrictions to all $\D_\psi$ are projective.  The hypothesis on $G$, and Lemma \ref{lem:536}, ensure that at all $1$-parameter subgroups $\psi$ the algebra $\D_\psi$ admits an (alternative) cocommutative Hopf structures.  Hence, by Lemma \ref{lem:947}, $V_K$ is projective over $\D_\psi$ if and only it $V_K$ is projective at all $\pi$-points for $\D_\psi$.  Taking this information together, we see that $V$ is projective over $\D$ if and only if $V$ is projective at all $\pi$-points $\alpha:K[t]/(t^p)\to \D_K$ for $\D$.
\end{proof}

\begin{remark}\label{rem:strong}
There are ways to define the $\pi$-support $\Pi(\D)_M$ of an arbitrary (possibly infinite-dimensional) $\D$-module $M$ so that Theorem \ref{thm:pi_proj_check} remains valid at arbitrary $M$.  However, it is unclear whether or not the equivalence relation on $\pi$-points $K[t]/(t^p)\to \D_K$ defined via finite-dimensional representations agrees with the analogous one defined via arbitrary modules (cf.\ \cite[Theorem 4.6]{friedlanderpevtsova07}).  Rather, in the language of \cite{friedlanderpevtsova07}, it is unclear whether equivalent $\pi$-points are in fact \emph{strongly} equivalent.  So we do not know if the support space $\Pi(\D)_M$ can be defined in such a way that depends only on the classes $[\alpha]$ of $\pi$-points, and not the $\pi$-points themselves.  We therefore leave a discussion of $\pi$-point support for infinite-dimensional modules to some later investigation.
\end{remark}

\subsection{Tensor product properties and comparison with cohomological support}

As discussed in subsection \ref{sect:compare}, one can read the material of Section \ref{sect:tensor_fd} through the alternate lens of $\pi$-point support.  In particular, the arguments of Section \ref{sect:tensor_fd} imply that $\pi$-point support behaves well with respect to tensor products, and also agrees with cohomological support (cf.\ \cite{friedlanderpevtsova05,friedlanderpevtsova07}).
\par

We have the following.

\begin{proposition}\label{prop:998}
For any infinitesimal group scheme $G$, and embedded $1$-parameter subgroup $\psi:\mbb{G}_{a(s)}\to G$, $\pi$-point support for $\D_\psi$ satisfies the tensor product property
\[
\Pi(\D_\psi)_{V\ot W}=\Pi(\D_\psi)_V\cap \Pi(\D_\psi)_W.
\]
\end{proposition}

\begin{proof}
If $\Pi(\D_\psi)$ is empty there is nothing to prove.  If $\Pi(\D_\psi)$ is non-empty, then $\pi$-point support for $\D_\psi$ is identified with cohomological support, via Lemma \ref{lem:947}.  So the result follows by the tensor product property for cohomological support provided in the (proof of) Proposition \ref{prop:tpp_Dp}.
\end{proof}

An important reading of Proposition \ref{prop:998} is the following: given a $\pi$-point $\alpha:K[t]/(t^p)\to \D_\psi$, and $\D_\psi$-representations $V$ and $W$, the restriction $\res_\alpha(V\ot W)$ is non-projective if and only both $\res_\alpha(V)$ and $\res_\alpha(W)$ are non-projective.  Since $\pi$-point support for the global algebra $\D$ is itself defined via $\pi$-points for the varying $\D_\psi$, the following result is immediate.

\begin{theorem}\label{thm:tpp_Pi}
For any infinitesimal group scheme $G$, $\pi$-point support for $\D$ satisfies the tensor product property
\[
\Pi(\D)_{V\ot W}=\Pi(\D)_V\cap \Pi(\D)_W.
\]
\end{theorem}

Finally, when $G$ is a Frobenius kernel in a sufficiently nice algebraic group $\mbb{G}$, we find that $\pi$-point support is identified with cohomological support.

\begin{theorem}\label{thm:Pi}
Suppose that $G$ is a Frobenius kernel in an algebraic group $\mbb{G}$, and that $\mbb{G}$ admits a quasi-logarithm.  Then the map $w:\Pi(\D)\to |\D|$ of Proposition \ref{prop:973} is a homeomorphism, and restricts to a homeomorphism $\Pi(\D)_V\to |\D|_V$ for all finite-dimensional $\D$-representations $V$.
\end{theorem}

\begin{proof}
Let $w:\Pi(\D)\to |\D|$ denote the map $[\alpha]\mapsto p_\alpha$ of Proposition \ref{prop:973}.  Under the above hypotheses Lemma \ref{lem:536} tells us that all $\D_\psi$ have non-vanishing $\pi$-support spaces $\Pi(\D_\psi)$.  So Lemma \ref{lem:947} tells us that $\pi$-supports and cohomological supports are identified for all $\D_\psi$.
\par

Suppose we have two $\pi$-points $\alpha,\beta\in \Pi(\D)$ for which $p_\alpha=p_\beta$.  Let $V$ be any representation which is non-projective at $\alpha$.  Write explicitly $\alpha:K[t]/(t^p)\to \D_\psi\to \D_K$ and $\beta:K'[t]/(t^p)\to \D_{\psi'}\to \D_{K'}$.  Since, at any embedded $1$-parameter subgroups $\eta$, the composites
\[
\Pi(\D_\eta)\to \Pi(\D)\to |\D|\ \ \text{and}\ \ \Pi(\D_\eta)\overset{\cong}\to |\D_\eta|\to |\D|
\]
are both given by $[\eta]\mapsto p_\eta$, i.e.\ since the two composites agree, Proposition \ref{prop:restrict} ensures that $[\alpha]\in \Pi(\D_\psi)_{V_K}$ and $[\beta]\in \Pi(\D_{\psi'})_{V_{K'}}$.  Rather, both $\res_\alpha(V_K)$ and $\res_\beta(V_{K'})$ are non-projective.  Since $V$ was chosen arbitrarily, this shows $\alpha$ is equivalent to $\beta$.  So we see that $w$ is injective.  Surjectivity follows from Proposition \ref{prop:reconstruct}, applied to $V=k$.
\par

We understand now that $w:\Pi(\D)\to |\D|$ is a bijection of sets.  One argues similarly to see that each restriction $\Pi(\D)_V\to |\D|_V$ is a bijection.  Finally, since all basic closed subsets in $\Pi(\D)$ and $|\D|$ are realized as supports of finite-dimensional representation, we see that $w$ is in fact a homeomorphism.
\end{proof}

\subsection{Comparing with a universal $\pi$-point space}
\label{sect:Pi_ot}

Consider the Drinfeld double $\D$ of an arbitrary finite group scheme--or really any Hopf algebra.  We have a universal definition of ``$\pi$-points", from the perspective of classifying thick tensor ideals in the stable category.  Namely, we consider all flat algebra maps $\alpha:K[t]/(t^p)\to \D_K$ which satisfy the tensor product property:
\begin{center}
 ({\sf TPP}) $\res_\alpha(V_K\ot W_K)$ is non-projective if and only if both $\res_\alpha(V_K)$ and $\res_\alpha(W_K)$ are non-projective.
\end{center}
\par

As with our other classes of $\pi$-points, we consider the space $\Pi_\ot(\D)$ of equivalence classes of all such universal $\pi$-points, and topologize this space in the expected way.  To be clear, our equivalence relation for universal $\pi$-points is defined exactly as in Definition \ref{def:equiv}, where we simply replace ``$\pi$-point" with ``universal $\pi$-point" in the definition.  We have the supports
\[
\Pi_\ot(\D)_V=\{[\alpha]:\res_\alpha(V_K)\text{ is non-projective}\}
\]
and corresponding support theory $V\mapsto \Pi_\ot(\D)_V$.  
\par

One notes that the class of universal $\pi$-points is chosen in the coarsest possible was to ensure that the tensor product property
\[
\Pi_\ot(\D)_{V\ot W}=\Pi_\ot(\D)_V\cap \Pi_\ot(\D)_W
\]
holds, and to ensure that the support $\Pi_\ot(\D)_V$ depends only on the class of $V$ in the stable category.
\par

Now, if we specifically consider the Drinfeld double of an infinitesimal group scheme, Theorem \ref{thm:tpp_Pi} tells us that any $\pi$-point $\alpha:K[t]/(t^p)\to \D_K$ as in Definition \ref{def:1445} is a universal $\pi$-point.  Furthermore, the equivalence relations on $\pi$-points and universal $\pi$-points are exactly the same.  So we obtain a topological embedding $\iota:\Pi(\D)\to \Pi_\ot(\D)$ for which we have
\begin{equation}\label{eq:1593}
\Pi(\D)_V=\Pi(\D)\cap \Pi_\ot(\D)_V,
\end{equation}
simply by the definitions of these supports.

\begin{theorem}\label{thm:Pi_ot}
Suppose that $G$ is a Frobenius kernel in an smooth algebraic group $\mbb{G}$, and that $\mbb{G}$ admits a quasi-logarithm.  Then the inclusion $\iota:\Pi(\D)\to \Pi_\ot(\D)$ is a homeomorphism, and all of the restrictions $\iota_V:\Pi(\D)_V\to \Pi_\ot(\D)_V$ are also homeomorphisms.
\end{theorem}

\begin{proof}
Take $\msc{Z}=\stab(\D)$, and recall the isomorphism $w:\Pi(\D)\to |\D|$ of Theorem \ref{thm:Pi}.  The universal property of the Balmer spectrum \cite[Theorem 3.2]{balmer05}, and Theorem \ref{thm:tpp_Pi}, we have continuous maps to the Balmer spectrum $f_\pi:\Pi(\D)\to \Spec(\msc{Z})$ and $f_\ot:\Pi_\ot(\D)\to \Spec(\msc{Z})$ which are compatible, in the sense that $f_\ot\circ \iota=f_\pi$.  Similarly, the map $f_{\rm coh}:|\D|\to \Spec(\msc{Z})$ of Theorem \ref{thm:spec} is such that $f_{\rm coh}\circ w=f_\pi$.  Since $w$ and $f_{\rm coh}$ are homeomorphisms, by Theorems \ref{thm:spec} and \ref{thm:Pi}, we see that $f_\pi$ is a homeomorphism.  Since $f_\pi$ factors through $f_\ot$, we see that $f_\ot:\Pi_\ot(\D)\to \Spec(\msc{Z})$ is surjective.  We claim that this surjection is in fact a bijection.
\par

We have explicitly,
\[
\begin{array}{rl}
f_\ot(\alpha) &=\{V\in \msc{Z}: [\alpha]\notin \Pi_\ot(\D)_V\}\\
&=\{V\in \msc{Z}:\res_\alpha(V_K)\text{ is projective}\}
\end{array}
\]
\cite[Theorem 3.2]{balmer05}.  Hence $f_\ot(\alpha)=f_\ot(\beta)$ implies that any $\D$-representation with projective restriction along $\alpha$ also has projective restriction along $\beta$, and vice versa.  So, by definition, the two classes agree $[\alpha]=[\beta]$.  So we see that $f_\ot$ is injective, and therefore a bijection.  It follows that $\iota:\Pi(\D)\to \Pi_\ot(\D)$ is a bijection.  Since $\iota$ is a topological embedding, this bijection is furthermore a homeomorphism.  The fact that all of the restrictions $\Pi(\D)_V\to \Pi_\ot(\D)_V$ are homeomorphisms as well follows by the intersection formula \eqref{eq:1593}.
\end{proof}

We collect our results about the support theory $\Pi_\ot(\D)_\star$ from above to find the following, somewhat remarkable, corollary.

\begin{corollary}\label{cor:Pi_ot}
Fix $G$ as in Theorem \ref{thm:Pi_ot}, and $\D$ the corresponding Drinfeld double.  Then
\begin{enumerate}
\item $\D$ admits enough universal $\pi$-points, in the sense that a $\D$-representation $V$ is projective if and only if its restriction $\res_\alpha(V_K)$ along each universal $\pi$-point $\alpha:K[t]/(t^p)\to \D_K$ is projective.
\item The natural map $w:\Pi_\ot(\D)\to |\D|$, $[\alpha]\mapsto p_\alpha$, is a homeomorphism.  In particular, the universal $\pi$-point space $\Pi_\ot(\D)$ has the structure of a projective scheme.
\item Any flat map $\alpha:K[t]/(t^p)\to \D$ which satisfies the tensor product property {\rm(TPP)} is equivalent to one of the form required in Defintion \ref{def:1445}.
\end{enumerate}
\end{corollary}

Of course, the issue with the universal $\pi$-support $\Pi_\ot(\D)_\star$, in general, is that it is difficult to understand the space $\Pi_{\ot}(\D)$ explicitly, or even to understand when this space is non-empty.  So, one \emph{needs} a practical construction of $\pi$-points, as above, in order to populate $\Pi_\ot(\D)$ with enough points, and in order to see that this theory carries significant amounts of information.

\begin{remark}
For a general Hopf algebra $A$, we can define the universal $\pi$-point support theory $V\mapsto \Pi_\ot(A)_V$ exactly as above.  We make no claim that this theory is well-behaved, or even non-vacuous in general.  However, it is interesting that there are even any examples in characteristic $0$ where one has enough universal $\pi$-points.  For example, the results of \cite{pevtsovawitherspoon15} imply that the support theory $\Pi_{\ot}(A)_\star$ satisfies the conclusions of Corollary \ref{cor:Pi_ot} (1) \& (2), for $A$ a ``quantum elementary abelian group" over $\mbb{C}$.  Similarly, for finite group schemes, one can argue as in the proof of Theorem \ref{thm:Pi_ot} to see that the standard $\pi$-point support theory $\Pi(G)_\star$ and universal theory $\Pi_\ot(G)_\star$ agree.
\end{remark}

\subsection{Remaining questions}
\label{sect:Q}

At this point we have recorded a number of non-trivial results concerning $\pi$-points and support for Drinfeld doubles of (some) infinitesimal group schemes.  We record a number of remaining questions which the reader may consider.

\begin{question}
{(1)} Can one provide an intrinsic proof of the tensor product property of Theorem \ref{thm:tpp_Pi}, i.e.\ one which follows from a direct analysis of $\pi$-points, and does not reference an auxiliary support theory?  (Compare with \cite{friedlanderpevtsova05,pevtsovawitherspoon09,friedlander}.)

{(2)} Does the Drinfeld double of a general infinitesimal group scheme $G$ admit enough (universal) $\pi$-points, in the sense of Corollary \ref{cor:Pi_ot} (1)?

{(3)} Is there a reasonable extension of $\pi$-point support $\Pi(\D)_M$ to infinite-dimensional $M$?  In particular, does there exist such a definition which reproduces the tensor product property
\[
\Pi(\D)_{M\ot N}=\Pi(\D)_M\cap \Pi(\D)_N
\]
at arbitrary $M$ and $N$?
\end{question}

Of course, question (3) has to do with one's (in)ability to use $\pi$-point support in certain tensor triangular investigations, as in Section \ref{sect:thicK} and \cite{bensoniyengarkrause11,bensoniyengarkrausepevtsova18,balmerkrausestevenson19,balmer20} for example.

\bibliographystyle{abbrv}

\end{document}